\documentclass[journal,twoside,web]{ieeecolor}

\pdfoutput=1
\usepackage{generic}
\usepackage{graphicx}      
\usepackage{amsmath}
\usepackage{amssymb}
\usepackage{subcaption}
\usepackage{algorithm}
\usepackage{algpseudocode}
\usepackage{epstopdf}
\usepackage{placeins}
\usepackage{hyperref}
\usepackage{tikz}
\usepackage{enumerate}
\usetikzlibrary{calc}
\usepackage{generic}

\usetikzlibrary{shapes,arrows,calc}
\usetikzlibrary{positioning,fit,backgrounds}
\usetikzlibrary{tikzmark,shapes.geometric}

\allowdisplaybreaks
\newtheorem{theorem}{Theorem}
\newtheorem{assumption}{Assumption}
\newtheorem{proposition}{Proposition}

\newtheorem{remark}{Remark}
\newtheorem{definition}{Definition}

\DeclareMathAlphabet{\pazocal}{OMS}{zplm}{m}{n}
\newcommand{\tr}[0]{{^\intercal}}

\newcommand{\resp}[1]{{\color{black} #1}}

\begin{document}

\title{\resp{An explicit dual control approach for constrained reference tracking of uncertain linear systems}}

\author{Anilkumar Parsi, Andrea Iannelli, Roy S. Smith, \IEEEmembership{Fellow, IEEE}
\thanks{This work was supported by the Swiss National Science Foundation under Grant 200021\_178890.}
\thanks{The authors are with the Automatic Control Laboratory, Swiss Federal Institute of Technology (ETH Z\"urich), 8092 Zurich, Switzerland (e-mail:
aparsi/iannelli/rsmith@control.ee.ethz.ch).}
}

\maketitle
\begin{abstract}
    A finite horizon optimal tracking problem is considered for linear dynamical systems subject to parametric uncertainties in the state-space matrices and exogenous disturbances. A suboptimal solution is proposed using a model predictive control (MPC) based explicit dual control approach which enables active uncertainty learning. A novel algorithm for the design of robustly invariant online terminal sets and terminal controllers is presented. Set membership identification is used to update the parameter uncertainty online. A predicted worst-case cost is used in the MPC optimization problem to model the dual effect of the control input. The cost-to-go is estimated using contractivity of the proposed terminal set and the remaining time horizon, so that the optimizer can estimate future benefits of exploration. 
    The proposed dual control algorithm ensures robust constraint satisfaction and recursive feasibility, and navigates the exploration-exploitation trade-off using a robust performance metric. 
\end{abstract}


\section{Introduction}
\IEEEPARstart{I}{n} control problems involving unknown plants, the importance of balancing probing actions aimed at inferring properties of the system (exploration) and greedy decisions that maximize the performance based on the current knowledge (exploitation) has been well understood since the pioneering work in adaptive control \cite{Astrom_book}. Feldbaum \cite{feldbaum1960dual} indicated that an optimal adaptive control law was characterized by the right combination of these two, often conflicting, properties, and coined the expression \emph{adaptive dual control systems}.
This optimal trade-off can be characterized by means of the principle of optimality, and thus the exact solution to the dual control problem is given by stochastic dynamic programming (DP) \cite{MESBAH_review}. Because of its well-known curse of dimensionality, tractable solutions to this problem are not available except for trivial cases. This motivated approximate solutions which can be broadly divided into two groups, namely implicit and explicit dual control methods \cite{Filatov_survey}. The former considers different types of approximations of the stochastic DP problems, e.g. using the wide-sense property from adaptive control \cite{Bar-Shalom_TAC73,Bayard_IJACSP_10} or making use of approximate DP methods such as policy iteration \cite{book_Powell_ADP}. To apply approximate DP algorithms to unknown systems,
a hyper-state is often considered, which augments the system's state with the vector of unknown parameters \cite{Lee_JPC_09,ARCARI20208105}. The main issue generally associated with implicit approaches is scalability, as evidenced by their application to date to only very low-order problems. On the other hand, explicit methods relax the original optimal control problem (OCP) into a more tractable form, and then reformulate it so that the solution exhibits the desired dual properties. Even though the optimality gap with respect to the original optimal dual policy is generally unknown, this often represents a much more computationally tractable solution and thus is by far the most established of the two approaches \cite{Milito_tac82,Pronzato_TAC96,DiCairano_CDC_14,HEIRUNG2017340,soloperto2019dual,Iannelli_LCSS_20}. In an important class of explicit methods, the dual properties are introduced in an application-oriented fashion \cite{LARSSON20161}. That is, probing actions that decrease a measure of the model uncertainty based on control-oriented measures (e.g. a robust performance criterion) are incentivized. 

An important consideration for the control of uncertain systems is the guaranteed satisfaction of state and input constraints, which is also referred to as safety in literature \cite{hewing2020learning}. Such safety guarantees can be achieved using model predictive control (MPC), an optimization-based technique where a control objective can be maximized while ensuring constraint satisfaction \cite{borrelli2017predictive}. 
When the plant is unknown, robust approaches need to be employed to guarantee stability and constraint satisfaction in the face of the uncertainty \cite{langson2004}, \cite{kouvaritakis2015model}. If the initial uncertainty is large, robustness will come at the cost of performance. Thus, it is beneficial to update the uncertainty set online using measurements gathered from the system.
An adaptive MPC (AMPC) scheme was developed in \cite{lorenzen2017adaptive} and \cite{lorenzen2019} for regulation of linear systems described by state-space models subject to polytopic parametric uncertainty. It uses a tube MPC approach to ensure robust constraint satisfaction and updates the uncertainty using set-membership identification \cite{milanese1991}. 
\resp{Adaptive MPC has also been proposed for nonlinear systems using a Lipschitz based method in \cite{adetola2011robust}, and a general framework has been developed in \cite{kohler2020robust}}. Nevertheless, there is no active uncertainty learning \resp{in these methods} since actions are designed for the purpose of exploitation only. \resp{ Another class of optimization-based methods for controlling unknown dynamical systems has recently emerged from the online learning community, where performance is commonly established by means of the concept of regret \cite{cohen2019}, \cite{hazan2020}. In the case of unknown dynamics, these methods typically consider unconstrained control problems and learning is effectively passive.}

In \cite{parsi2020active}, \resp{we} proposed an extension to \cite{lorenzen2019} which introduces dual actions in AMPC by means of a predicted state tube while guaranteeing constraint satisfaction. The problem considered was again only a regulation one, and was implicitly framed in an infinite horizon setting. For linear systems with model uncertainties, regulating the state to the origin is easier than performing setpoint tracking. This is because the input setpoint is assumed to be known, thereby simplifying the design of MPC ingredients. Although this issue is partially addressed in the tracking MPC literature, most methods assume that the system is affected only by either additive disturbances  \cite{maeder2009linear}, \cite{zeilinger2014soft}, \cite{di2015reference}, \cite{pereira2016robust} or parametric uncertainities \cite{pannocchia2004robust}, \cite{zhu2015adaptive}. Two related works are also \cite{HEIRUNG2017340,soloperto2019dual}, which proposed a dual adaptive tracking MPC algorithm with uncertainty in the measurement equation only. Moreover, studying the dual effect of the control input is much more relevant in a finite horizon setting \cite{Iannelli_LCSS_20}. This is because there is limited time for learning the uncertainty and probing actions have decreasing rewards with time. Thus, estimating the benefits of exploratory control actions is difficult in this setting, compared to an infinite horizon setting where it is almost always beneficial to reduce model uncertainty. We therefore view finite horizon problems as a more meaningful, albeit difficult, domain in which the exploration-exploitation trade-off can be studied.

The goal of this paper is to propose an explicit dual control approach for finite horizon optimal tracking control of uncertain systems described by parametric state-space models.  There are two main technical contributions. \resp {First, we propose a novel tracking adaptive MPC formulation in Section \ref{Sec:Tracking} for plants subject to parametric uncertainty in addition to exogenous disturbances. The proposed algorithm estimates the uncertain setpoints online and uses an online terminal set and a control law dependent on the estimated setpoints. The second contribution of the work is to leverage the viewpoint of MPC as an approximate DP method \cite{hewing2020learning, BERTSEKAS2005310} to obtain an application-oriented dual control algorithm to solve the finite horizon optimal control problem. That is, the proposed algorithm navigates the exploration-exploitation trade-off using a performance based metric as described in Section \ref{Sec:Exploration}. } To the best of the authors' knowledge, this is thus the first dual adaptive MPC scheme for optimal reference tracking of systems with exogenous disturbances and model mismatch in the dynamics.



\subsection{Problem formulation and methodology}
Consider an uncertain, discrete-time, linear system of the form
\begin{equation}\label{eq:Dynamics}
    x_{t+1} = A(\theta) x_t + B(\theta) u_t + w_t,
\end{equation}
where  $x_t \in \mathbb{R}^n$ represents the state, $u_t \in \mathbb{R}^m$ the control input and $w_t \in \mathbb{R}^n$ the additive disturbance at the time step $t$. The state space matrices $A(\theta)$ and $B(\theta) $ have the parametric description
\begin{align} \label{eq:Parameterization}
\begin{split}
    A(\theta) = A_0 + \displaystyle\sum_{i=1}^{p} A_i [\theta]_i, \quad  B(\theta) = B_0 + \displaystyle\sum_{i=1}^{p} B_i [\theta]_i,
\end{split}
\end{align}
where $\theta \in \mathbb{R}^p$ is an unknown constant parameter, $[\theta]_i$ refers to the $i^{\text{th}}$ element in $\theta$ and $\{A_i,B_i\}_{i=0}^{p}$ are known matrices modeling structured uncertainty. The true value of $\theta=\theta^*$ is known to lie inside the bounded set
\begin{equation}\label{eq:ParameterBounds}
    \Theta := \{\theta \in \mathbb{R}^p | H_{\theta} \theta \le h_{\theta} \},
\end{equation}
where $H_{\theta} \in\mathbb{R}^{n_\theta \times p}$ and $h_{\theta} \in \mathbb{R}^{n_{\theta}}$. The disturbance $w_t$ has unknown statistical properties, but always lies within the bounded set
\begin{equation}\label{eq:DisturbanceBounds}
    \mathbb{W}:= \{w\in\mathbb{R}^n |  H_{w} w \le h_{w} \},
\end{equation}
where $ H_w \in \mathbb{R}^{n_w\times n}, h_w\in\mathbb{R}^{n_w}$. 
\resp{
\begin{assumption}\label{As:StateMeasurements}
	The state of the system is perfectly measured at each time step.
\end{assumption}
}
The output to be tracked is given by
\begin{equation}\label{eq:Outputs}
y_t = C x_t,
\end{equation}
where $y_t\in\mathbb{R}^{n_y}$.
The states and inputs of the system are required to satisfy the constraints
\begin{equation} \label{eq:Constraints}
\mathbb{Z} = \left\{(x_t,u_t) \in \mathbb{R}^n \times \mathbb{R}^m \bigr|  F x_t + G u_t \le \mathbf{1}\right\},
\end{equation}
where $ \mathbb{Z} $ is a compact set, and $ F \in \mathbb{R}^{n_c \times n} $ and  $ G \in \mathbb{R}^{n_c \times m} $ are given matrices. Given an initial state $\overset{\circ}{x}$, the ideal objective would be to compute the control \resp{policy} $\Pi:= \{u_t = \pi_t(x_t)\}_{t=0}^{T}$ which solve the following finite horizon optimal control problem (FHOCP) 
\begin{subequations}\label{eq:FHOCP}
\begin{align}
    \displaystyle\min_{\Pi} \max_{w_t} \quad \sum_{t=0}^{T} ||Q(y_{t}{-}r_{t})||_{\infty} {+} ||R(&\pi_{t}(x_t){-}u^*_t)||_{\infty}, \label{eq:FHOCP1}\\
    \text{s.t.}  \quad A(\theta^*) x_t + B(\theta^*) \pi_t(x_t) + w_t &= \resp{   x_{t+1}}, \label{eq:FHOCP2}\\
                        F x_t + G \pi_t(x_t) &\le \mathbf{1}, \label{eq:FHOCP3}\\
                        y_t &= C x_t,\\
                        \resp{A(\theta^*) x_t^* + B(\theta^*) u_t^*} &\resp{= x_{t+1}^*,}  \label{eq:FHOCP6}\\
                        \resp{C x_t^*} &\resp{= r_t , \quad   \forall t \in \mathbb{N}_0^T, } \label{eq:FHOCP7} \\
                        x_0 = \overset{\circ}{x}, \quad x_0^{*} &= \overset{\circ}{x}, \label{eq:FHOCP5} 
\end{align}
\end{subequations} 
where $T$ is the length of the finite control horizon, $Q\in\mathbb{R}^{n\times n}$ and $R\in\mathbb{R}^{m\times m}$ are positive definite matrices, and $\mathbb{N}_0^T$ denotes the set of integers from 0 to $T$. In addition, $\{r_t\}_{t=0}^{T}$ is a known reference trajectory for the output, $\{u^*_t\}_{t=0}^{T}$ is a sequence of unknown input setpoints, \resp{and $\{x^*_t\}_{t=0}^{T}$ represents an unknown desired state trajectory satisfying \eqref{eq:FHOCP6},\eqref{eq:FHOCP7}.} \resp{For legibility, the state evolution predicted by the FHOCP is represented using the same vaiables  $\{x_t\}_{t=0}^{T}$ used for the evolution of the true state in \eqref{eq:Dynamics}}. \resp{Solving} \eqref{eq:FHOCP} is difficult because $\theta^*$ is not known, which results in an uncertain cost function \eqref{eq:FHOCP1}, state evolution \eqref{eq:FHOCP2} \resp{ and setpoints \eqref{eq:FHOCP6}}. 
Hence, the objective of this work is to compute suboptimal policies to minimize the cost \eqref{eq:FHOCP1}  for the uncertain system \eqref{eq:Dynamics} while satisfying the constraints \eqref{eq:Constraints}.
\resp{Without any loss in robustness of constraint satisfaction, the setpoint can be estimated using a nominal parameter estimate $ \bar{\theta} $.} In order to ensure robustness of constraint satisfaction in face of the uncertainty, one can solve the robust optimization problem
\begin{align}\label{eq:FHOCP_minmax}
\begin{split}
    \displaystyle\min_{\Pi} \max_{\substack{\theta\in \Theta,\\ w_t\in\mathbb{W}}} \: \sum_{t=0}^{T} ||Q(y_{t}-r_{t})||_{\infty} {+} ||R(&\pi_{t}(x_t)-u^*_t)||_{\infty}, \\
    \text{s.t.}  \quad  A(\theta) x_t + B(\theta) \pi_t(x_t) + w_t &=\resp{ x_{t+1},} \\
                        F x_t + G \pi_t(x_t) &\le \mathbf{1}, \\
                        y_t &= C x_t, \\
                        \resp{A(\bar{\theta} ) x_t^* + B(\bar{\theta} ) u_t^*} &\resp{= x_{t+1}^*,} \\
                        \resp{C x_t^*} &\resp{= r_t , \quad   \forall t \in \mathbb{N}_0^T, } \\
                        x_0 = \overset{\circ}{x}, \quad x_0^{*} &= \overset{\circ}{x}.
\end{split}
\end{align}
Nevertheless, the robustness guarantees provided by \eqref{eq:FHOCP_minmax} could also result in conservatism when $\Theta$ is large. This conservatism can be reduced by using the measurements \resp{from the system} to reduce the size of the uncertainty set \eqref{eq:ParameterBounds}. To this aim, set membership identification \cite{milanese1991} is used in this work (Section \ref{Sec:ParameterIdentification}). In this approach, the uncertainty set $\Theta$ is updated by constructing non-falsified parameter sets from state measurements and control inputs. This results in a dual effect for the control, where the control inputs affect both tracking quality and identification performance. It is well known that such problems can be highly non-convex, even for scalar systems with small control horizons $T$ \cite{klenske2016dual}. For this reason, the control problem \eqref{eq:FHOCP_minmax} is reformulated as described below. 

The control policies $\pi_t(x_t)$ are restricted to be time invariant and affine in the state variable, and are formulated by using an estimate of the state and input setpoint, and a prestabilizing feedback gain (Section \ref{Sec:ControlParameterization}). At each time step $k$, a feasible input sequence for a shorter time interval $[k,k+N-1]$ (called the prediction horizon, with length $N$) is computed using MPC. In addition, a tube MPC approach \cite{rakovic2012homothetic} is used to construct a robust state tube consisting of all the states that the system could reach in the prediction horizon under the defined control policy and an updated parameter set (Section \ref{Sec:RobustTube}). 
Tube MPC controllers ensure that the last set of the state tube lies inside a robust invariant set called the terminal set, which is in general computed offline. 
In Section \ref{Sec:TerminalSets}, an online algorithm to design robustly invariant terminal sets is proposed, which uses the updated parameter set to improve the tracking performance. 

To enable dual control, the effect of control input on the identification is modeled in Section \ref{Sec:PredParSet}, and a predicted state tube is constructed in Section \ref{Sec:PredictedTube}. The robust cost function in \eqref{eq:FHOCP_minmax} is then approximated by a worst-case cost over the predicted state tube (Section \ref{Sec:PredictedCost}). 
This captures the dual effect of the control using an application-oriented metric \cite{LARSSON20161}. Moreover, the cost to be incurred after the end of the prediction horizon (hereafter called cost-to-go) is approximated by using the remaining time in the control horizon and the contractivity of the terminal set. This allows a novel trade-off between exploration and exploitation which will depend on, among other things, the number of time steps remaining in the control horizon. The dual adaptive MPC (AMPC) algorithm and its properties are discussed in Section \ref{Sec:DAMPC}.

\subsection{Notation}
The sets of real numbers and non-negative real numbers are denoted by $ \mathbb{R} $ and $ \mathbb{R}_{\ge0} $ respectively. 
For a vector $ b $, $||b||_k$ represents the $k-$norm for $k\in\{2,\infty\}$. The $ i ^{th}$ row of a matrix $ A $ is denoted by $ [A]_{i} $, \resp{and $  \otimes $ represents the Kronecker product}. The dimensions of matrices and vectors are not always specified when they can be inferred from the context.  \resp{For any real scalar-valued function $ J $, $ \displaystyle\max_{h \in \mathbb{H}} J(h)$ refers to the maximization of $ J $ over the set $ \mathbb{H} $. }
The Minkowski sum of two sets $ A$ and $ B $ is denoted by $ A \oplus B $, and $ \mathbf{1} $ denotes a column vector of appropriate length whose elements are equal to 1.  The convex hull of the elements of a set S is represented by co\{S\}. The notation $ a_{l|k} $ denotes the value of $ a $ at time step $ k+l $ computed at the time step $ k $. The identity matrix of size $n\times n$ is denoted by $I_n$. 
%

\section{Tracking formulation for AMPC}\label{Sec:Tracking}
In this section, the constraints in FHOCP  \eqref{eq:FHOCP_minmax} are reformulated using MPC and set membership identification is introduced to reduce the size of the uncertainty set online. An existing tube MPC approach is modified to enable setpoint tracking by defining the control law as a function of the reference trajectory and \resp{the} estimates of the system's uncertain parameters. Offline and online designs of terminal sets are then proposed. 

\subsection{Parameter identification}\label{Sec:ParameterIdentification}
Set-membership identification is a technique used to identify systems affected by bounded noise with unknown statistical properties \cite{milanese1991}. In this method, the measurement data is used to update a set of feasible parameters, denoted as $\Theta_t$. In order to ensure that the number of constraints in $\Theta_t$ does not increase at every time step, it is defined as
\begin{equation} \label{eq:Theta_k_def}
\Theta_t :=  \{\theta \in \mathbb{R}^p| H_{\theta} \theta\le h_{\theta_t}\},
\end{equation}
where $ h_{\theta_t}$ initialized with $h_{\theta}$, and updated online using measurement data. To perform the update, a set of non-falsified parameters is constructed using the last $\tau$ measurements as
\begin{align} \label{eq:SimpleNonfalsified}
\Delta_{t} &:= \biggr\{
\theta \in \mathbb{R}^{p}\: \biggr| \: -H_w D_i \theta  \le h_w+H_w d_{i+1}, \forall i \in \mathbb{N}_{t-\tau}^{t-1}
\biggr\}\nonumber\\ 
&= \left\{\theta \in \mathbb{R}^{p}\: |\: H_\Delta \theta \le h_\Delta \right\},
\end{align}
where $D_i:= D(x_i,u_i) \in\mathbb{R}^{n\times p}$ and $d_{i+1}\in \mathbb{R}^{n}$ are 
\setlength\arraycolsep{2pt}
\begin{align}\label{eq:Dtdt}
\begin{split}
    D_i =& \begin{bmatrix}
         A_1 x_i {+} B_1 u_i, & \ldots, & A_p x_i {+} B_p u_i
    \end{bmatrix},  \\
    d_{i+1} :=& \: A_0 x_t + B_0 u_t - x_{i+1}, \qquad \forall i \in \mathbb{N}_{t-\tau}^{t-1}.
\end{split}
\end{align}
Using \eqref{eq:SimpleNonfalsified}, $\Theta_t$ is updated such that
\begin{equation}\label{eq:Theta_k_update}
 \Theta_t \supseteq \Theta_{t-1} \cap \Delta_t 
\end{equation}
is satisfied. This is ensured by calculating $  h_{\theta_t} $ according to \resp{the $ {n_\theta} $ linear programs
( $ \forall  i \in \mathbb{N}_1^{n_\theta} $) }
\begin{align}\label{eq:Theta_k_LP}
\begin{split}
[h_{\theta_t}]_i \: =\: &\underset{\theta}{\resp{\text{max}}} \quad  [H_\theta]_{i} \theta\\
 \text{s.t. } \quad  
 \begin{bmatrix}
    H_\theta \\ H_\Delta
\end{bmatrix} \theta
 &\le \begin{bmatrix}
  h_{\theta_{t-1}} \\ h_{\Delta}
 \end{bmatrix}.\\
\end{split}
\end{align}
\resp{
The optimization problems in \eqref{eq:Theta_k_LP} ensure that the set update satisfies $\Theta_t \subseteq \Theta_{t-1} $, because the maximum value of the objective is upper bounded by $ [h_{\theta_{t-1}}]_i$ in the constraints. Thus, the parameter identification in \eqref{eq:Theta_k_LP} results in sets $\Theta_t$ of non-increasing size. 
}
\resp{Note that the the quality of identification improves by increasing the number of hyperplanes chosen to represent the set $\Theta_t$. However, this results in an increase in the computational complexity of the MPC optimization problem later formulated in Section \ref{Sec:DAMPC}.}

\resp{In addition, a parameter estimate $\bar{\theta}_k$ is computed, which will be used to define the control law and the setpoints in the cost function.} At each time step, this estimate is computed \resp{using} the center of the maximum volume $\ell_2$-norm ball that can be inscribed inside the set $\Theta_k$, \resp{also known as the Chebyshev center of the set \cite{boyd2004convex}}. The estimate, $\bar{\theta}_k$, is \resp{the optimal value of $ \theta_c $ in the following quadratic program}
\begin{align}\label{eq:ChebyshevMax}
\begin{split}
\displaystyle\max_{\resp{r_c},\theta_c}\quad \resp{r_c} - \mu ||\theta_c&-\bar{\theta}_{k-1}||_2^2, \\
\text{s.t.} \qquad [H_\theta]_i \theta_c + \resp{r_c}||[H_\theta ]_i||_{\resp{2}} &\le h_{\theta_k}, \quad \forall i \in \mathbb{N}_1^{n_\theta},
\end{split}
\end{align}
\resp{where the second term in the cost function ensures that the problem has a unique solution, and the parameter $ \mu>0 $ must be chosen to trade-off rate of change in the parameter estimate with the distance from the Chebyshev center of the set.} 

\resp{ 
	\begin{remark}
		The Chebyshev center of a set reduces the deviation of the estimate from the true parameter based on all the points in parameter set. Alternatively, one can use a least mean squares or a recursive least squares filter \cite{haykin2014Adaptive} to obtain the parameter estimate.
	\end{remark}
}
\subsection{Control policy parameterization}\label{Sec:ControlParameterization}
In order to formulate a computationally tractable optimization problem, a receding horizon approach is used, where at each time step $k$, a feasible input sequence is computed for the prediction horizon $[k,k+N-1]$. Note that $k$ will be used as the time index while formulating the MPC optimization problem at time step $k$. Within the prediction horizon, the control policies $\pi_k(x_k)$ are restricted to be affine functions of $x_k$, with the linear term being time-invariant. The policies also depend on estimates of the state and input setpoints, which are computed using the parameter estimate $\bar{\theta}_k$.

In order to estimate the state and input setpoints $(x_k^*,u_k^*)$ along the prediction horizon, the following set of linear constraints can be defined along the lines of \resp{\eqref{eq:FHOCP6},\eqref{eq:FHOCP7}}
\begin{equation}\label{eq:SetpointSubspace}
    \begin{bmatrix}
    A(\bar{\theta}_k) & B(\bar{\theta}_k) \\
    C & 0
    \end{bmatrix} \begin{bmatrix}
    \bar{x}_{l|k,k}\\\bar{u}_{l|k,k}
    \end{bmatrix} = \begin{bmatrix}
    \bar{x}_{l+1|k,k} \\ r_{k+l}
    \end{bmatrix} , \quad \forall l \in \mathbb{N}_{0}^{N-1},
\end{equation}
where the subscripts in $\bar{x}_{l|k,k}$ denote that the estimate of the setpoint at time step $k+l$ (first subscript) was computed using the parameter estimate $\bar{\theta}_k$ (second subscript). The unknown setpoints $(x_{l+k}^*,u_{l+k}^*)$ can be approximated by $(\bar{x}_{l|k,k},\bar{u}_{l|k,k})$ within the MPC horizon, i.e., for $ l \in \mathbb{N}_{0}^{N-1} $. In addition, the setpoints at the end of the prediction horizon are defined as
\begin{equation}\label{eq:SetpointSubspaceTerm}
\begin{bmatrix}
A(\bar{\theta}_k\resp{)} {-}I_n & B(\bar{\theta}_k) \\
C & 0
\end{bmatrix} \begin{bmatrix}
\bar{x}_{N|k,k}\\\bar{u}_{N|k,k}
\end{bmatrix} = \begin{bmatrix}
0 \\ r_{k+N}
\end{bmatrix}.
\end{equation}
The constraints \eqref{eq:SetpointSubspaceTerm} define the setpoints $ (\bar{x}_{N|k,k}, \bar{u}_{N|k,k}) $ as an equilibrium state and input, which will be used in the design of online terminal sets in Section \ref{Sec:TerminalSets}. 

\begin{assumption}\label{As:FullRankEstimate}
    The system satisfies $m\ge n_y$, i.e., the system has at least as many inputs as the desired outputs to be tracked. Additionally, for all $\theta\in \Theta$, the system satisfies
    \begin{align}\label{eq:FullRankEstimate}    
    \begin{split}
    \text{rank} \left(\begin{bmatrix}
    A_{\text{st}}({\theta}) & B_{\text{st}}({\theta}) \\
    C_{\text{st}} & 0
    \end{bmatrix} \right) &{=} N(n{+}n_y), \\
   	\text{rank} \left(\begin{bmatrix}
       A({\theta}) {-} I_n & B({\theta}) \\
       C & 0
       \end{bmatrix} \right) &{=} n{+}n_y,
   \end{split}
    \end{align}
    where the matrices $  A_{\text{st}}({\theta}) \in \mathbb{R}^{Nn\times Nn},  B_{\text{st}}({\theta})  \in \mathbb{R}^{Nm\times Nm}$ and  $ C_{\text{st}}  \in \mathbb{R}^{Nn_y\times Nn_y} $ are defined as
    \begin{align}
    \begin{split}
    	 A_{\text{st}}({\theta}) &= \begin{bmatrix}
    	 A(\theta) & {-I_n} & 0 & & \dots & 0 \\
    	 0& A(\theta) & {-I_n} & & \dots & 0 \\
    	 \vdots &  & \ddots & & & \vdots \\
    	 0& &\dots &  & A(\theta) & {-}I_n  \\
    	 0 &  & \dots & & & A(\theta)
    	 \end{bmatrix}, \\
    	 B_{\text{st}}({\theta}) &= I_N \otimes B(\theta),\quad  C_{\text{st}} = I_N \otimes C,
  	 \end{split}
    \end{align}
    and $ \otimes $ denotes the Kronecker product operator.
\end{assumption}
\resp{ Note that Assumption \ref{As:FullRankEstimate} is made so that there exists a state and input setpoint sequence satisfying \eqref{eq:SetpointSubspace}, \eqref{eq:SetpointSubspaceTerm} for any given reference trajectory and any parameter in the set $\Theta_0$. Similar assumptions }  are made in tracking MPC literature considering systems with model uncertainty in the system matrices \cite{pannocchia2004robust}. 

\resp{The control inputs in the prediction horizon are now parameterized as}
\begin{equation}\label{eq:InputParameterization}
u_{l|k}\resp{(x)} = K(\resp{x}-\bar{x}_{l|k,k}) + v_{l|k}, \quad \forall l\in \mathbb{N}_{0}^{N-1},
\end{equation}
where $ \{v_{l|k}\}_{l=0}^{N-1} $ are decision variables in the MPC optimization problem and $K\in \mathbb{R}^{m\times n}$ is a prestabilizing feedback matrix which satisfies the following assumption.
\begin{assumption}\label{As:Feedback}
	\resp{The parameter set $\Theta_0$ is such that there exists a feedback gain $K$ which asymptotically stabilizes $ A_{\text{cl}}(\theta)  = A( \theta) + B(\theta)K, \: \forall \theta \in \Theta$. }
\end{assumption}
\resp{Such a stabilizing feedback gain }$ K $ can be computed using standard robust control techniques, for example, as suggested in \cite{kothare1996}.

\subsection{Robust state tube}\label{Sec:RobustTube}
The robust tube MPC approach proposed in \cite{lorenzen2017adaptive} is now modified for the tracking problem using the setpoints $\bar{x}_{l|k,k}$ in the prediction horizon. A robust state tube is constructed to contain all the states that the system can reach under the control law \eqref{eq:InputParameterization}, for all realizations of $\theta\in\Theta_k$ and $w\in \mathbb{W}$. Specifically, the state tube $\{\mathbb{X}_{l|k}\}_{l=0}^{N}$ satisfies
\begin{align}\label{eq:SetDynamics}
\mathbb{X}_{0|k} \ni \{x_{k}\},  \quad
\mathbb{X}_{l+1|k} &\supseteq A(\theta)\mathbb{X}_{l|k}  \oplus B(\theta)u_{l|k} \oplus \mathbb{W},\\
& \qquad  \forall \theta \in \Theta_k, \quad l \in \mathbb{N}_{0}^{N-1}. \nonumber
\end{align}
To ensure that the representation of the sets $\mathbb{X}_{l|k}$ is tractable, they are defined as a sequence of homothetic sets \cite{rakovic2012homothetic}, \resp{that is, as a translation and scaling of a predefined polytope $\mathbb{X}_0 $. This polytope is designed offline, and has the description  }
\begin{align}
 \mathbb{X}_0 := \{x| H_x x \le \mathbf{1}\} = \text{co}\{x^{1},x^{2},\ldots,x^{q}\},
\end{align}
where $ H_x \in \mathbb{R}^{n_x\times n} $ and  $\{x^{1},x^{2},\ldots,x^{q}\} $ are the vertices of $\mathbb{X}_0$. 
\resp{ In the online optimization, the center of $\mathbb{X}_0$ is translated using $ z_{l|k} \in \mathbb{R}^{n}$ and its size is scaled by $ \alpha_{l|k} \in \mathbb{R}_{\ge0}$ to define the set $\mathbb{X}_{l|k} $. Thus, the propagation of the state tube is defined by the optimization variables $\{z_{l|k} ,\alpha_{l|k}\}_{l=0}^{N}$ according to}
\begin{align}\label{eq:StateTubeParameterization}
\mathbb{X}_{l|k} &= \{z_{l|k}\} \oplus \alpha_{l|k} \mathbb{X}_{0} \quad = \{x| H_x (x-z_{l|k}) \le \alpha_{l|k} \mathbf{1}\}\nonumber\\
&= \{z_{l|k}\} \oplus \alpha_{l|k} \text{co}\{x^{1},x^{2},\ldots,x^{q}\}. 
\end{align}

Using the above parameterization, the constraints \eqref{eq:Constraints} and set dynamics \eqref{eq:SetDynamics} can be reformulated as linear constraints in terms of the optimization variables. For this purpose, the following notation is used
\begin{align}\label{eq:x_jlk}
\begin{split}
x_{l|k}^{j} &= z_{l|k} + \alpha_{l|k} x^{j}, \quad u_{l|k}^{j} {=} K(x_{l|k}^{j} - \bar{x}_{l|k,k}) {+} v_{l|k},  \\
 D_{l|k}^{j} &= D(x_{l|k}^{j},u_{l|k}^{j}), \quad  d_{l|k}^{j} {=} A_0 x_{l|k}^{j} {+} B_0 u_{l|k}^{j} - z_{l+1|k},
\end{split}
\end{align}
where $  j \in \mathbb{N}_{1}^{q} ,l \in \mathbb{N}_{0}^{N-1} $. Note that unlike the definition in \eqref{eq:Dtdt} where $ D_t$ and $d_t $ are functions of known states and inputs, the quantities $ D_{l|k}^{j}$ and $ d_{l|k}^{j} $ linearly depend on the optimization variables.  Additionally, the vectors $ \bar{f}$ and $\bar{w} $ are computed offline such that for $  i \in \mathbb{N}_{1}^{n_c},  j \in \mathbb{N}_{1}^{n_x}  $,
\begin{equation}\label{eq:f_bar}
\begin{split}
 [\bar{f}]_{i} &= \underset{x\in \mathbb{X}_0}{\resp{\text{max}}} [F+GK]_i x,\quad [\bar{w}]_{j} = \underset{w\in \mathbb{W}}{\resp{\text{max}}}  \: [H_x]_j w.
\end{split}
\end{equation}
The following proposition reformulates the robust constraints and set-dynamics  as affine constraints.
\begin{proposition}[\resp{\cite[Proposition 9]{lorenzen2017adaptive}}]\label{Pr:SetDynamics}
	Let the state tube $ \{\mathbb{X}_{l|k}\}_{l=0}^{N} $ be parameterized according to \eqref{eq:StateTubeParameterization}. Then, the constraints \eqref{eq:Constraints} and set-dynamics \eqref{eq:SetDynamics} are satisfied if and only if $\forall  j{\in} \mathbb{N}_{1}^{q} $, $ l{\in} \mathbb{N}_{0}^{N-1}$, there exist matrices $ \Lambda_{l|k}^{j} \in \mathbb{R}^{n_x\times n_\theta}_{\ge0}$ with non-negative elements  such that 
	\begin{subequations}\label{eq:lambdaConstraints}
		\begin{align}
		(F+GK)z_{l|k} + Gv_{l|k} \resp{- GK\bar{x}_{l|k,k}} &+ \alpha_{l|k}\bar{f} \le \mathbf{1},\\
		-H_x z_{0|k} -\alpha_{0|k}\mathbf{1} &\le -H_x x_k ,		\label{eq:x0Constraint}\\
		\Lambda_{l|k}^{j} h_{\theta_k} + H_x d_{l|k}^{j} -\alpha_{l+1|k} \mathbf{1} &\le -\bar{w} ,\label{eq:InclusionIneq}\\
		H_x D_{l|k}^{j} &= \Lambda_{l|k}^{j} H_{\theta} \label{eq:InclusionEqual}.
		\end{align} 
	\end{subequations}
\end{proposition}

\subsection{Terminal set}\label{Sec:TerminalSets}
The system's state must not be driven into the regions of the state space which could result in constraint violations after the prediction horizon. This can be guaranteed \resp{by ensuring the MPC algorithm is recursively feasible}, which is generally achieved using a robustly invariant terminal set and a terminal control law.

\begin{definition}[Robust Invariance {\cite{kouvaritakis2015model}} ]\label{De:RobustInvariance}
 A set $\mathbb{X}$ is said to be robustly invariant under the dynamics \eqref{eq:Dynamics} and a terminal control law $\pi(x)$ if $A(\theta)x + B(\theta)\pi(x) + w \in \mathbb{X}$ and $(x,\pi(x)) \in \mathbb{Z}$ for all $x \in \mathbb{X}, \theta \in \Theta$ and $w\in\mathbb{W}$. 
\end{definition}

In this paper, we use a similar approach to that followed in \cite{parsi2020active} and impose terminal constraints to make sure that $\mathbb{X}_{N|k}$ lies inside an invariant set. As the controller in \cite{parsi2020active} is designed to perform regulation, the center $z_{N|k}$ \resp{was set as} the origin and $\alpha_{N|k}$ is bounded by a precomputed constant $\bar{\alpha}$. However, a different strategy must be employed for reference tracking to account for the uncertainty in the input setpoint. We propose two different methods, offline and online, for the design of terminal sets.
\subsubsection{Offline terminal set design}
To design the terminal set offline, the initial parameter estimate $\bar{\theta}_0$ is used to generate setpoint estimates $(\bar{x}_{N|k,0},\bar{u}_{N|k,0})$ for $k\in \mathbb{N}_0^{T-N}$. These setpoints might not lie in the feasible region $\mathbb{Z}$. For this reason, the terminal set to be used at time step $k$ is defined using an artificial setpoint $(\tilde{x}_{N|k,0},\tilde{u}_{N|k,0})$ and an upper bound $\tilde{\alpha}_{N|k} \in \mathbb{R}_{\ge0}$ on $ \alpha_{N|k} $. The size of the terminal set is thus specified by $\tilde{\alpha}_{N|k}$, which is large when the setpoint $(\tilde{x}_{N|k,0},\tilde{u}_{N|k,0})$ is far from the constraints \eqref{eq:Constraints}. \resp{The terminal control law is chosen to be of the form $ \pi^{\pazocal{T}}(x) = K(x-\Tilde{x}_{N|k}) + \Tilde{u}_{N|k}$}. Hence $\tilde{x}_{N|k,0},\tilde{u}_{N|k,0}$ and $\tilde{\alpha}_{N|k}$ are computed offline by solving \resp{the following $T{-}N{+}1$ quadratic programs (for $ k \in \mathbb{N}_{0}^{T-N} $)}
\begin{subequations}\label{eq:OfflineTerminalSet}
\begin{align}
\min_{\substack{\tilde{x}_{N|k,0},\tilde{u}_{N|k,0},\tilde{\alpha}_{N|k}, \\\bar{x}_{N|k,0},\bar{u}_{N|k,0} }}   
|| \tilde{x}_{N|k,0} {-} \bar{x}_{N|k,0}||_2^2 +  &\nonumber \\[-12pt]
 || \tilde{u}_{N|k,0} {-} &\bar{u}_{N|k,0}||_2^2 - \xi \tilde{\alpha}_{N|k}, \nonumber \\[3pt]
\text{s.t.} \:
        \begin{bmatrix}
        A(\bar{\theta}_0)-I & B(\bar{\theta}_0) \\
        C & 0 
        \end{bmatrix} \begin{bmatrix}
        \bar{x}_{N|k,0} \\ \bar{u}_{N|k,0}
        \end{bmatrix}& = \begin{bmatrix}
        0 \\ r_{k+N}
        \end{bmatrix},  \label{eq:OfflineTerminalSet1}\\
        F(\tilde{x}_{N|k,0} {+} \tilde{\alpha}_{N|k} x^j) + G(\tilde{u}_{N|k,0} + \tilde{\alpha}_{N|k}&K  x^j) \le \mathbf{1},\label{eq:OfflineTerminalSet2}\\
        H_x(A(\theta^i)(\tilde{x}_{N|k,0} {+} \tilde{\alpha}_{N|k}x^j) + \quad & \nonumber \\ \resp{B(\theta^i)(\tilde{u}_{N|k,0}+\tilde{\alpha}_{N|k}Kx^j)) }
         {+} &\bar{w} \le \tilde{\alpha}_{N|k}\mathbf{1}, \label{eq:OfflineTerminalSet3}\\ 
        \forall &i \in \mathbb{N}_{1}^{q_\theta}, j\in \mathbb{N}_{1}^{q}, \nonumber
\end{align}
\end{subequations}
where the cost function tries to minimize the distance between the artificial and estimated setpoints, and the constant $\xi>0$ trades off this distance with the size of the terminal set. Note that $\{\theta^i\}_{i=1}^{q_\theta}$ represent the vertices of the set $\Theta_0$. The constraints \eqref{eq:OfflineTerminalSet2} ensure that the terminal set is within $\mathbb{Z}$, and \eqref{eq:OfflineTerminalSet3} imposes the robust invariance condition. The set $\mathbb{X}_{N|k}$ is constrained to lie inside the invariant region by setting:
\begin{equation}\label{eq:OfflineTerminalConstraints}
    z_{N|k} = \tilde{x}_{N|k,0}, \quad \alpha_{N|k} \le \tilde{\alpha}_{N|k}.
\end{equation}
Although the terminal sets designed in \eqref{eq:OfflineTerminalSet} account for the reference trajectory and can be tuned to find the largest terminal sets, such an offline design would be conservative for two reasons. The first reason is that the identification step \eqref{eq:Theta_k_LP} updates the parameter set and can result in a larger terminal region, which is not leveraged in this approach. The second reason is that choosing the size of the terminal set apriori can lead to poor performance. For example, choosing a small $\xi$ results in a small terminal set and reduces the feasible region of MPC.

\subsubsection{Online terminal set design}
The reasons discussed above motivate an online design where $z_{N|k}$ is set equal to an optimization variable $\Tilde{x}_{N|k}\in \mathbb{R}^{n}$, whose distance from $ \bar{x}_{N|k,k} $ can be minimized by the MPC optimizer. Additionally, $\alpha_{N|k}$ is upper bounded by a scalar optimization variable $\tilde{\alpha}_{N|k}\in \mathbb{R}_{\ge0}$, which is constrained such that the terminal set $\mathbb{X}_{N|k}^{\pazocal{T}} := \Tilde{x}_{N|k} \oplus \tilde{\alpha}_{N|k}\mathbb{X}_{0}$ is robustly invariant for the system \eqref{eq:Dynamics} under the terminal control law
\begin{equation}\label{eq:TerminalControl}
    \resp{\pi_{k}^{\pazocal{T}}(x) = K(x-\Tilde{x}_{N|k}) + \Tilde{u}_{N|k}}
\end{equation}
where $\Tilde{u}_{N|k}{\in} \mathbb{R}^{m}$ is an optimization variable. 
Using the notation
\begin{align}\label{eq:x_jNk}
\begin{split}
\tilde{x}_{N|k}^{j} &= \tilde{x}_{N|k} + \tilde{\alpha}_{N|k} x^{j}, \quad \tilde{u}_{N|k}^{j} = K(x_{N|k}^{j} - \tilde{x}_{N|k}) + \Tilde{u}_{N|k},  \\
 D_{N|k}^{j} &= D(\tilde{x}_{N|k}^{j},\tilde{u}_{N|k}^{j}), \quad  d_{N|k}^{j} = A_0 \tilde{x}_{N|k}^{j} + B_0 \tilde{u}_{N|k}^{j} - \tilde{x}_{N|k},\quad 
\end{split}
\end{align}
for $j{\in}\mathbb{N}_1^{q}$, robust invariance of $\mathbb{X}_{N|k}^{\pazocal{T}}$ can be formulated as
\begin{subequations}\label{eq:RPITerminalSet}
\begin{align}
	F\tilde{x}_{N|k} + G\Tilde{u}_{N|k} + \tilde{\alpha}_{N|k}\bar{f} &\le \mathbf{1}, \label{eq:RPI1}\\
	\tilde{\Lambda}_{N|k}^{j} h_{\theta_k} + H_x d_{N|k}^{j} -\tilde{\alpha}_{N|k} \mathbf{1} &\le -\bar{w}, \label{eq:RPI2}\\
	H_x D_{N|k}^{j} &= \tilde{\Lambda}_{N|k}^{j} H_{\theta}\label{eq:RPI3},
\end{align} 
\end{subequations}
where $j\in \mathbb{N}_1^{q}$ and $ \tilde{\Lambda}_{N|k}^{j} \in \mathbb{R}^{n_x\times n_\theta}_{\ge0}$ are Lagrange multiplier variables, similar to \eqref{eq:lambdaConstraints}. The feasibility of the constraints in \eqref{eq:RPITerminalSet} will be addressed in Proposition \ref{Pr:TerminalSet}. The optimization variables defining $\mathbb{X}_{N|k}$ are constrained as
\begin{equation}\label{eq:OnlineTerminalConstraints}
    z_{N|k} = \tilde{x}_{N|k}, \quad \alpha_{N|k} \le \tilde{\alpha}_{N|k}.
\end{equation}
Thus, the online terminal sets guarantee robust invariance for parameters in $\Theta_k$ instead of $\Theta$, and also provide flexibility to increase their size depending on the feasibility of the reference trajectory. In view of these advantages, the subsequent sections formulate the MPC problem using online terminal sets. 

\begin{remark}
Note that a terminal set is used in the MPC approximation of \eqref{eq:FHOCP} for the time steps $k$ where $k{+}N{<}T$ in order to retain feasibility after time step $k{+}N$. When $ k{+}N{\ge}T $, i.e., towards the end of the control horizon, the terminal constraints can be omitted from the MPC optimization problem and the value of $N$ can be appropriately shortened.
\end{remark}

\subsubsection{Design of $\mathbb{X}_0$}
An \resp{important aspect of the MPC algorithm} is the design of $\mathbb{X}_0$, which is used to propagate the set dynamics as well as parameterize the terminal set. Thus, an essential design condition for $\mathbb{X}_0$ is that \eqref{eq:RPITerminalSet} can admit a feasible solution. Using Assumption \ref{As:Contractivity}, the following proposition shows that the feasibility of the terminal constraints can be guaranteed. 
\begin{assumption}\label{As:Contractivity}
The set $\mathbb{X}_0$ \resp{is such that there exists a scalar $\lambda_c\in [0,1)$ satisfying
\begin{subequations}
\begin{align}\label{eq:Contractivity}
   H_x \bigr( A(\theta)x + B(\theta)Kx\bigr) &\le \lambda_c \mathbf{1}, \quad \forall \theta\in \Theta, x \in \mathbb{X}_0,\\
   \lambda_c &\le 1- \bar{w}\bar{\bar{f}} \label{eq:LambdaBound}
\end{align}
\end{subequations}
where $\bar{\bar{f}} = \underset{i}{\text{max}} [\bar{f}]_i, i \in \mathbb{N}_{1}^{n_c} $. }
\end{assumption}
\resp{The conditions in  \eqref{eq:Contractivity} imply that the set $ \mathbb{X}_0 $ is $\lambda$-contractive with contractivity factor $\lambda_c$ for the undisturbed dynamics of the system \eqref{eq:Dynamics} under the control law $u=Kx$. }
\vspace{3pt}
\begin{proposition}\label{Pr:TerminalSet}
    If Assumption \ref{As:Contractivity} is satisfied, then the terminal conditions \eqref{eq:RPITerminalSet} admit a feasible solution.
\end{proposition} 
\begin{proof}
The contractivity of $\mathbb{X}_{0}$ \eqref{eq:Contractivity} implies
\begin{align}\label{eq:TerminalSetFeas1}
    H_x A_{\text{cl}}(\theta) x^j \le \lambda_c \mathbf{1}, \quad  \theta \in \Theta, j \in \mathbb{N}_1^q .
\end{align}
Using Definition \ref{De:RobustInvariance}, the conditions for robust invariance of $\mathbb{X}_{N|k}^{\pazocal{T}}$ can be written as
\begin{subequations}
\begin{align}
    F x &+ G \pi_{N|k}^{\pazocal{T}}(x) \le \mathbf{1} \label{eq:TerminalSetFeas1a},\\
    H_x(A(\theta)x + B(\theta)\pi_{N|k}^{\pazocal{T}}(x) &+ w - \tilde{x}_{N|k}) \le \tilde{\alpha}_{N|k} \mathbf{1},\label{eq:TerminalSetFeas1b}\\ &\forall x \in \mathbb{X}_{N|k}, \theta \in \Theta, w \in \mathbb{W}.\nonumber
\end{align}
\end{subequations}
Substituting the control law \eqref{eq:TerminalControl} and using \eqref{eq:f_bar}, \eqref{eq:TerminalSetFeas1a} can be rewritten as \eqref{eq:RPI1}. Similarly, \eqref{eq:TerminalSetFeas1b} can be reformulated as
\begin{subequations}
\begin{align}
H_x\bigr(A_{\text{cl}}(\theta)(x &- \tilde{x}_{N|k}) + A(\theta)\tilde{x}_{N|k} + B(\theta)\Tilde{u}_{N|k}  \label{eq:TerminalSetFeas2a}  \\
&\quad -\tilde{x}_{N|k}\bigr) + \bar{w} \le \tilde{\alpha}_{N|k} \mathbf{1}, \nonumber \\
\iff  H_x\bigr(A_{\text{cl}}(\theta)&(\tilde{\alpha}_{N|k} x^j) +A(\theta)\tilde{x}_{N|k} + B(\theta)\Tilde{u}_{N|k}  \label{eq:TerminalSetFeas2b}  \\ 
&-\tilde{x}_{N|k}\bigr) + \bar{w} \le \tilde{\alpha}_{N|k} \mathbf{1} , \quad \forall j\in\mathbb{N}_{1}^{q}.\nonumber
\end{align}
\end{subequations}
Note that \eqref{eq:TerminalSetFeas2b} is equivalent to the pair of equations \eqref{eq:RPI2},\eqref{eq:RPI3}, where the Lagrange multipliers $\tilde{\Lambda}_{N|k}^j$ are used to reformulate a maximization over $\theta$ to ensure robustness. The equivalence holds due to strong duality. Using \eqref{eq:TerminalSetFeas1}, the condition \eqref{eq:TerminalSetFeas2b} is always satisfied if
\begin{align}\label{eq:TerminalSetFeas3}
    \bar{w} +  H_x\bigr(A(\theta)\tilde{x}_{N|k} + B(\theta)\Tilde{u}_{N|k}  -\tilde{x}_{N|k}\bigr)\le  \tilde{\alpha}_{N|k} (\mathbf{1} - \lambda_c).
\end{align}
It can be seen that the conditions \eqref{eq:RPI1} and  \eqref{eq:TerminalSetFeas3} admit  $\tilde{x}_{N|k} = \tilde{u}_{N|k} = 0$, and $ \tilde{\alpha}_{N|k} = 1/\bar{\bar{f}}$ as a feasible solution.
\end{proof}
Although the above proof provides a feasible terminal set with center at the origin, the parameter $\lambda_c$ is a design variable which can be tuned to allow other feasible solutions. The offline design procedure to be followed is described in Algorithm 1. Given the system dynamics, a control gain $K$ is first designed. Then, a $ \lambda$-contractive set $\mathbb{X}_0$ with a chosen $\lambda_c$ is constructed using the algorithm given in \cite[\resp{p.~3}]{didier2021}. 
An important point to consider is that the design of $\mathbb{X}_0$ affects the control performance through the choice of $\lambda_c$, and the number of vertices $q$ and hyperplanes $n_x$ in $\mathbb{X}_0$ affect the computational complexity 
as seen  in \eqref{eq:lambdaConstraints}. 

\setlength{\textfloatsep}{2pt}
\begin{algorithm}[h]
	\caption{Offline design procedure}\label{Alg:Offline} 
	\begin{algorithmic}[1]
	    \Statex \textbf{Input}: $\Theta$, \textbf{Output}: $K,\mathbb{X}_0$
        \State Compute robustly stabilizing gain $K$ (e.g., see \cite{kothare1996})
		\State Choose a contractivity rate $\lambda_c \in [0,1)$
		\State Compute $\mathbb{X}_0$ as suggested in \cite[\resp{p.~3}]{didier2021}
		\State Compute the constants $\bar{w},\bar{f},\bar{\bar{f}}$
		\If {$\lambda_c \le 1- \bar{w}\bar{\bar{f}}$} \label{Alg:lambdaCondition}
            \State STOP
        \Else
            \State Return to step 1, redesign $K,\mathbb{X}_0$
        \EndIf
	\end{algorithmic}
\end{algorithm}
\begin{remark}
Note that the design of polytopic contractive sets is a difficult problem, which has been extensively studied for linear systems  \cite{kolmanovsky1998theory}, \resp{\cite{blanchini2008set}, \cite{rubin2018computation}}. \resp{The upper bound \eqref{eq:LambdaBound} in Assumption \ref{As:Contractivity} implies that the condition in step \ref{Alg:lambdaCondition} of Algorithm \ref{Alg:Offline}  is satisfied. If this condition} is not satisfied, one must redesign $K$ and choose a new $\lambda_c$. Even though there is no systematic way to ensure an improved design at the next iteration, one heuristic that worked in practice was to choose smaller $\lambda_c$, and impose $\lambda$-contractivity on the ellipsoidal region stabilized by $K$ \cite{kothare1996}. 
\end{remark}

\section{Dual control using constraint predictions}\label{Sec:Exploration}
The robust state tube  defined in Section \ref{Sec:Tracking} ensures that the constraints are satisfied for any parameter in $\Theta_k$. This section models the effect of the control inputs on the identification scheme to enable dual control in closed loop. For this purpose, we define predicted variables and sets which are analogous to the ones defined in Section \ref{Sec:Tracking}. These predicted quantities are denoted with a hat ($ \hat{~} $). 

\subsection{Predicted parameter set} \label{Sec:PredParSet}
\resp{The parameter identification step \eqref{eq:Theta_k_update} connects the control inputs to the parameter sets in closed loop. This link can be leveraged within the MPC problem by predicting the future state measurements  as $ \hat{x}_{l|k} $ according to 

\begin{align}\label{eq:xhat}
\begin{split}
\hat{x}_{l+1|k} &= A(\bar{\theta}_k) \hat{x}_{l|k} + B(\bar{\theta}_k) \hat{u}_{l|k} ,\\
\hat{u}_{l|k} &= K (\hat{x}_{l|k} - \bar{x}_{l|k,k}) + v_{l|k}, \quad \forall l\in\mathbb{N}_{0}^{N_{\theta} -1}  \\
\hat{x}_{0|k} &= x_k,
\end{split}
\end{align}
where $ \bar{\theta}_k $ is the parameter estimated by \eqref{eq:ChebyshevMax}, and $ N_{\theta} \in \mathbb{N}_{2}^{N} $ is the \emph{lookahead} horizon. Using these predicted measurements, the predicted non-falsified sets $\{\hat{\Delta}_{l|k}\}_{l=1}^{N_{\theta}}$ are defined as 
\begin{align}\label{eq:Delta_hat}
\hat{\Delta}_{ l|k} &:= \{\theta \: | \hat{x}_{t+1|k} {-} A(\theta) \hat{x}_{t|k} {-} B(\theta) \hat{u}_{t|k} \in \mathbb{W},  \forall t \in \mathbb{N}_{0}^{l -1} \}. 
\end{align}
The hyperplanes defining the sets $ \hat{\Delta}_{ l|k} $ depend affinely on the control input variables $ v_{l|k} $ as seen in \eqref{eq:xhat}. This dependence is omitted for clarity of presentation. 
A sequence of predicted parameter sets $ \{\hat{\Theta}_{l|k}\}_{l=1}^{N_{\theta}} \subseteq \Theta_k $ can now be defined as
\begin{align}\label{eq:PredictedParameterSet}
\hat{\Theta}_{l|k} := \Theta_k \cap \hat{\Delta}_{l|k} 
			    &= \: \{\theta \in \mathbb{R}^{p} | \hat{H}_{\theta_{l|k}} \theta \le  \hat{h}_{\theta_{l|k}} \}.
\end{align}
An illustration of the parameter estimate, the predicted constraints and parameter set is shown in Fig \ref{fig:parameterSetExample}. 

\begin{remark}
	Note that the lookahead horizon $ N_{\theta} $ must have a minimum value of 2. This is because, when $ N_{\theta} =1 $ the predicted non-falsified set $ \hat{\Delta}_{ N_{\theta}|k} $ is described by the first control input $ u_k $ and the known state measurement $ x_k $. This means that, if some uncertain parameters only affect the matrix $A(\theta)$, the predicted constraints on these parameters cannot be characterized as a function of the MPC optimization variables. As a result, there would be no dual action of the control input for these parameters. Additionally, it can be seen that as $ N_{\theta} $ increases, the length of the input sequence affecting the parameter estimation increases, allowing for improved identification. However, this also results in higher computational costs due to the non-convex constraints in the optimization problem, as seen in Section \ref{Sec:PredictedTube}. 
\end{remark}
}

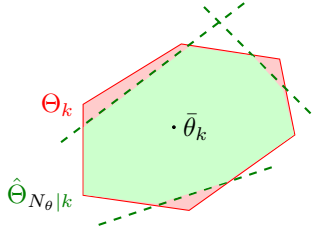
\begin{figure}
	\centering
		\centering
		\begin{tikzpicture}
		
		\draw [red,thick](0,0) -- (1.4,-0.2) -- (2.8,0.8) -- (2.6,1.8) -- (1.3,2) -- (0,1.2) -- cycle;
		\node[anchor=east,red] at (0,1.2){$ \Theta_k $};
		\fill [red!20!white](0,0) -- (1.4,-0.2) -- (2.8,0.8) -- (2.6,1.8) -- (1.3,2) -- (0,1.2) -- cycle;
		
		\draw [dashed,green!50!black,thick](-0.3,0.7) -- (2.2,2.6);
		\draw [dashed,green!50!black,thick](1.6,2.5) -- (3.1,1);
		\draw [dashed,green!50!black,thick](0.2,-0.4) -- (2.6,0.4);
		\node[anchor=east,green!50!black] at (0,0){$ \hat{\Theta}_{\resp{N_{\theta}}|k} $};
		
		\fill[green!20!white] (0,0) -- (0,0.9)-- (1.4,1.97) -- (2.25,1.85) -- (2.65,1.4) -- (2.8,0.8) -- (1.9,0.17) -- (0.97,-0.12) --cycle;
		
		\coordinate (x0) at (1.2,0.9);
		\filldraw (x0) circle (0.5pt) node[anchor=west] {$ \bar{\theta}_k $};
		\end{tikzpicture}
		\caption{ The parameter set $ \Theta_k $ is bounded by the constraints shown in red, and the parameter estimate $ \bar{\theta}_k $ lies inside this set. The predicted bounds are shown as dashed green lines, which depend on the control input \resp{variables $ v_{l|k} $}. The green shaded region shows the resulting predicted parameter set $ \hat{\Theta}_{\resp{N_{\theta}}|k} $.}
		\label{fig:parameterSetExample}
\end{figure}
\subsection{Predicted state tube}\label{Sec:PredictedTube}
In order to capture the effect of $u_k$ on the performance via the cost, a predicted state tube is constructed such that \resp{
\begin{align}\label{eq:PredictedStateTube}
\begin{split}
&\hat{\mathbb{X}}_{0|k} \ni \{x_{k}\},  \\
&\hat{\mathbb{X}}_{l+1|k} \supseteq A(\theta)\hat{\mathbb{X}}_{l|k}  \oplus B(\theta)u_{l|k} \oplus \mathbb{W}, \: \forall  \theta \in \hat{\Theta}_{l|k}, l \in \mathbb{N}_{0}^{N_{\theta}-1} \\
&\hat{\mathbb{X}}_{l+1|k} \supseteq A(\theta)\hat{\mathbb{X}}_{l|k}  \oplus B(\theta)u_{l|k} \oplus \mathbb{W}, \: \forall  \theta \in \hat{\Theta}_{N_{\theta}|k},l \in \mathbb{N}_{N_{\theta}}^{N-1}.
\end{split}
\end{align}
}
Note that the same control law \eqref{eq:InputParameterization} is used to propagate the robust and predicted state tubes. However, the parameter sets in \eqref{eq:SetDynamics} and \eqref{eq:PredictedStateTube} are different. The sets $\hat{\mathbb{X}}_{l|k} =\{\hat{z}_{l|k}\} \oplus \hat{\alpha}_{l|k}\mathbb{X}_0$ are parameterized using  $ \hat{z}_{l|k} \in \mathbb{R}^n, \hat{\alpha}_{l|k}\in\mathbb{R}_{\ge0}$ for all $l\in\mathbb{N}_{0}^{N} $. The following notation is now defined for $  j \in \mathbb{N}_{1}^{q} ,l \in \mathbb{N}_{0}^{N-1} $
\setlength\arraycolsep{0pt}
\begin{equation}\label{eq:x_til_jlk}
\begin{array}{rll}
\hat{x}_{l|k}^{j} &= \hat{z}_{l|k} + \hat{\alpha}_{l|k} x^{j}, \quad &\hat{d}_{l|k}^{j} = A_0 x_{l|k}^{j} + B_0 u_{l|k}^{j} - \hat{z}_{l+1|k},  \\
\hat{D}_{l|k}^{j} &= D(\hat{x}_{l|k}^{j},\hat{u}_{l|k}^{j}), \quad & \hat{u}_{l|k}^{j} = K(\hat{x}_{l|k}^{j} \resp{-\bar{x}_{l|k,k}}) + v_{l|k}.
\end{array}
\end{equation}
The next proposition formulates the dynamics of the predicted state tube as affine constraints.
\begin{proposition}\label{Pr:PredictedSetDynamics}
The predicted state tube $ \{\hat{\mathbb{X}}_{l|k}\}_{l=0}^{N}$ satisfies the set-dynamics \eqref{eq:PredictedStateTube} if and only if, for all $ j\in \mathbb{N}_{1}^{q} $ and $ l\in \mathbb{N}_{0}^{N-1}$ there exists $ \hat{\Lambda}_{l|k}^{j} \in \mathbb{R}^{n_x\times (n_\theta+n_w)}_{\ge0}$ such that 
\begin{subequations}\label{eq:lambda_tilConstraints}
\begin{align}
-H_x \hat{z}_{0|k} - \hat{\alpha}_{0|k}\mathbf{1} &\le -H_x x_k, \\
\hat{\Lambda}_{l|k}^{j} \hat{h}_{\theta_{\resp{l}|k}} + H_x \hat{d}_{l|k}^{j} -\hat{\alpha}_{l+1|k} \mathbf{1} &\le -\bar{w},\label{eq:lambda_tilConstraints2}\\
H_x \hat{D}_{l|k}^{j} &= \hat{\Lambda}_{l|k}^{j} \hat{H}_{\theta_{l|k}}.\label{eq:lambda_tilConstraints3}
\end{align}
\end{subequations}  
\end{proposition}
The constraints \eqref{eq:lambda_tilConstraints2} and \eqref{eq:lambda_tilConstraints3} are bilinear in the optimization variables, because \resp{the hyperplanes defining $ \hat{\Delta}_{l|k} $} are affinely dependent on the control input variables $ v_{l|k} $, as seen in \eqref{eq:xhat},\eqref{eq:Delta_hat}. An illustration of the robust and predicted state tubes is shown in Fig. \ref{Fig:StateTubes}. It can be seen that predicted state tube always lies within the robust state tube  $(\hat{\mathbb{X}}_{l|k}\subseteq\mathbb{X}_{l|k})$, because $\hat{\Theta}_{\resp{l}|k} \subseteq \Theta_k $ according to \eqref{eq:PredictedParameterSet}.

\begin{figure}
	\centering
	\begin{tikzpicture}[scale=0.65]
	\draw[blue,very thick](0,5.2) -- (12,5.2);
	
	\draw[red,thick](1.1,1) -- (1.4,1.8) -- (1.6,2.6) -- (1.8,3) -- (2,3.3) -- (2.2,3.5) -- (2.4,3.6) -- (2.6,3.7) -- (2.8,4) -- (3,4.2) -- (3.2,4.4) -- (3.4,4.6) -- (3.6,4.8) -- (3.8,4.9) -- (4,5) -- (4.2,5.1) -- (4.4,5.15) -- (4.6,5.17) -- (4.8,5.15) -- (5,5.17);
	\draw[red,thick](1.1,1) -- (1.4,1) -- (1.6,0.97) -- (1.8,0.95) -- (2,0.925) -- (2.2,0.9) -- (2.4,0.95) -- (2.6,1.1) -- (2.8,1.22) -- (3,1.42) -- (3.2,1.6) -- (3.4,1.8) -- (3.6,1.95) -- (3.8,2.15) -- (4,2.2) -- (4.2,2.3) -- (4.4,2.35) -- (4.6,2.4) -- (4.8,2.45) -- (5,2.5);
	
	\fill[red!20!white](1.1,1) -- (1.4,1.8) -- (1.6,2.6) -- (1.8,3) -- (2,3.3) -- (2.2,3.5) -- (2.4,3.6) -- (2.6,3.7) -- (2.8,4) -- (3,4.2) -- (3.2,4.4) -- (3.4,4.6) -- (3.6,4.8) -- (3.8,4.9) -- (4,5) -- (4.2,5.1) -- (4.4,5.15) -- (4.6,5.17) -- (4.8,5.15) -- (5,5.17) -- (5,2.5) -- (4.8,2.45) -- (4.6,2.4) -- (4.4,2.35) -- (4.2,2.3) -- (4,2.2) -- (3.8,2.15) -- (3.6,1.95) -- (3.4,1.8) -- (3.2,1.6) -- (3,1.42) -- (2.8,1.22) -- (2.6,1.1) -- (2.4,0.95) -- (2.2,0.9) -- (2,0.925) -- (1.8,0.95) -- (1.6,0.97) -- (1.4,1) -- cycle;

	\draw[green!50!black,thick,dashed](1.1,1) -- (1.4,1.4) -- (1.6,1.9) -- (1.8,2.2) -- (2,2.55) -- (2.2,2.75) -- (2.4,3) -- (2.6,3.2) -- (2.8,3.42) -- (3,3.6) -- (3.2,3.75) -- (3.4,3.8) -- (3.6,3.97) -- (3.8,4.1) -- (4,4.3) -- (4.2,4.4) -- (4.4,4.57) -- (4.6,4.7) -- (4.8,4.75) -- (5,4.8);
	
	\draw[green!50!black,thick,dotted](5,4.8) -- (12,4.3);
	\draw[green!50!black,thick,dotted](5,3.15) -- (12,3.7);
	
	\draw[green!50!black,thick,dashed](1.1,1) -- (1.4,1.15) -- (1.6,1.25) -- (1.8,1.37) -- (2,1.45) -- (2.2,1.52) -- (2.4,1.62) -- (2.6,1.8) -- (2.8,1.95) -- (3,2.15) -- (3.2,2.4) -- (3.4,2.52) -- (3.6,2.65) -- (3.8,2.8) -- (4,2.87) -- (4.2,2.95) -- (4.4,3) -- (4.6,3.05) -- (4.8,3.1) -- (5,3.15);

	\fill[green!20!white](1.1,1) -- (1.4,1.4) -- (1.6,1.9) -- (1.8,2.2) -- (2,2.55) -- (2.2,2.75) -- (2.4,3) -- (2.6,3.2) -- (2.8,3.42) -- (3,3.6) -- (3.2,3.75) -- (3.4,3.8) -- (3.6,3.97) -- (3.8,4.1) -- (4,4.3) -- (4.2,4.4) -- (4.4,4.57) -- (4.6,4.7) -- (4.8,4.75) -- (5,4.8) -- (5,3.15) -- (4.8,3.1) -- (4.6,3.05) -- (4.4,3) -- (4.2,2.95) -- (4,2.87) -- (3.8,2.8) -- (3.6,2.65) -- (3.4,2.52) -- (3.2,2.4) -- (3,2.15) -- (2.8,1.95) -- (2.6,1.8) -- (2.4,1.62) -- (2.2,1.52) -- (2,1.45) -- (1.8,1.37) -- (1.6,1.25) -- (1.4,1.15) -- cycle;
	
	\draw[black,dashed,thick](0,1.2) -- (2,1.2) -- (4,3.6) -- (6.6,3.6) -- (7.6,2) -- (12,2);
	\draw[black,thick](0,1) -- (1.1,1);
	\draw[blue,densely dotted](5,5.2) -- (5,2);
	
	\draw[->] [black,thin] (-0.75,0.5)  -- (12.5,0.5);
	\node[anchor=north,black] at (9,0.5){time};
	\draw[->] [black,thin] (-0.25,0)  -- (-0.25,5.5);
	\node[anchor=east,black,rotate=90] at (-0.5,4){state};
	
	\draw[black,densely  dotted](0,1) -- (0,0.5);
	\node[anchor=north,black] at (0,0.5){\small$0$};
	\draw[black,densely  dotted](1.1,1) -- (1.1,0.5);
	\node[anchor=north,black] at (1.1,0.5){\small$k$};
	\draw[black,densely  dotted](5,2) -- (5,0.5);
	\node[anchor=north,black] at (5,0.5){\small$k{+}N$};
	\draw[black,densely  dotted](12,2) -- (12,0.5);
	\node[anchor=north,black] at (12,0.5){\small$T$};
	
	\node[anchor=east,red] at (3,4.5){$\mathbb{X}_{l|k}$};
	\node[anchor=west,green!50!black] at (5,4.2){$\hat{\mathbb{X}}_{l|k}$};
	
	\end{tikzpicture}
	\caption{Evolution of the robust state tube $\mathbb{X}_{l|k}$ (red, solid) and the predicted state tube $\hat{\mathbb{X}}_{l|k}$ (green, dashed). The reference trajectory (black, dashed) and the state constraint (blue, solid) are also shown. The cost function is predicted using $\mathbb{\hat{X}}_{l|k}$, and cost-to-go is estimated using $\lambda$-contractivity (green, dotted).}
	\label{Fig:StateTubes}
\end{figure}
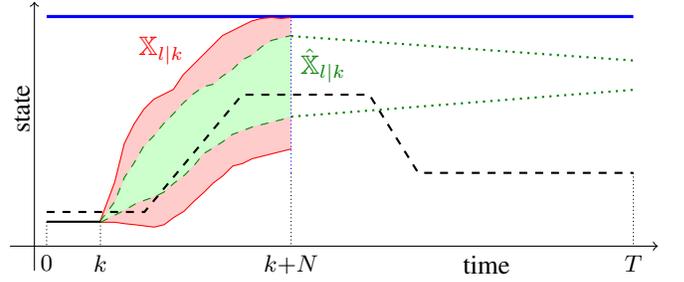
\subsection{Predicted cost function}\label{Sec:PredictedCost}
The predicted state tube captures the effect of the control input on the resulting uncertainty as well as the dynamics. However, its dependence on the parameter estimate $\bar{\theta}_k$ means that it cannot guarantee the trajectories to lie within the constraint set. A popular strategy used in safe-learning literature is to decouple the robustness and learning \cite{aswani2013provably}. This approach is used here, whereby the constraint satisfaction is ensured by the robust state tube, and the predicted state tube is used to define the AMPC cost function as the sum of stage costs and a terminal cost,
\begin{align}\label{eq:MPCcost}
\begin{split}
\Omega\bigr(\{\hat{\mathbb{X}}_{l|k}, \bar{x}_{l|k,k},\bar{u}_{l|k,k} \}_{l=0}^{N}, &\{v_{l|k}\}_{l=0}^{N-1},\tilde{u}_{N|k}  \bigr) := \\ 
\sum_{l=0}^{N-1} J(\hat{\mathbb{X}}_{l|k}&, \bar{x}_{l|k,k},\bar{u}_{l|k,k} ,v_{l|k}) + \\ 
J_\pazocal{T}(\hat{\mathbb{X}}_{N|k}&,\bar{x}_{N|k,k},\bar{u}_{N|k,k},\tilde{u}_{N|k})  .
\end{split}
\end{align}
The stage cost $J(\hat{\mathbb{X}}_{l|k}, \bar{x}_{l|k,k},\bar{u}_{l|k,k} ,v_{l|k})$ is defined as
\begin{align}\label{eq:MPC_stageCost}
    J(\hat{\mathbb{X}}_{l|k}, &\bar{x}_{l|k,k},\bar{u}_{l|k,k} ,v_{l|k}) = \max_{x \in \hat{\mathbb{X}}_{l|k}} ||Q(Cx - r_{k+l})||_{\infty} \\
    &+ ||R(K(x - \bar{x}_{l|k,k}) + v_{l|k} - \bar{u}_{l|k,k})||_{\infty},\nonumber
\end{align}
where the maximization over the predicted state tube is used to approximate the worst-case performance due to the uncertainty and disturbance. Using an epigraph reformulation, \eqref{eq:MPC_stageCost} can be described as a linear cost function \cite{boyd2004convex}. The terminal cost function $J_\pazocal{T}(\hat{\mathbb{X}}_{N|k},\bar{x}_{N|k,k},\bar{u}_{N|k,k},\tilde{u}_{N|k})$ must capture the cost-to-go from time step $k+N$ to $T$. This is approximated using the $\lambda$-contractivity of the terminal set as
\begin{align}\label{eq:MPC_terminalCost} 
\begin{split}
    J_\pazocal{T}(\hat{\mathbb{X}}_{N|k}&,\bar{x}_{N|k,k},\bar{u}_{N|k,k},\tilde{u}_{N|k}) = \\
     &\max_{x \in \hat{\mathbb{X}}_{N|k}} \beta(k+N)  \Bigr(||Q(Cx - r_{k+N})||_{\infty} \\
    &+ ||R(K(x {-} \bar{x}_{N|k,k}) + \tilde{u}_{N|k} - \bar{u}_{N|k,k})||_{\infty}\Bigr),
\end{split}
\end{align}
where $\beta(k) = (1-\lambda_c^{T-k})/(1-\lambda_c)$ is a time dependent factor. The cost-to-go is approximated using $\beta(k)$ under the assumption that the reference trajectory does not change and that the system is undisturbed in the time horizon $ [k+N,T] $. Thus, the cost function defined in \eqref{eq:MPCcost} estimates the  performance benefits of exploration within the prediction horizon in the stage cost \eqref{eq:MPC_stageCost}, and approximates the benefits after the prediction horizon in the terminal cost \eqref{eq:MPC_terminalCost} by using the contractivity and remaining length of the control horizon. \resp{Moreover, the cost function \eqref{eq:MPCcost} enables application-oriented dual control, since the optimizer can trade-off exploration with exploitation based on the estimated improvement of a robust performance metric. }


\begin{remark}	
The optimal control problem \eqref{eq:FHOCP} uses an infinity norm in the cost function, which enables to describe the MPC worst-case cost \eqref{eq:MPCcost} as a linear function using epigraph reformulation. Instead, if a more standard quadratic cost  is considered in \eqref{eq:FHOCP} and in MPC, the maximization of a quadratic function over the predicted state tube sets can be described using second order cone constraints. 
\end{remark}


\section{Dual adaptive MPC}\label{Sec:DAMPC}
\resp{In this section, we formulate the dual adaptive MPC algorithm using the robust and predicted state tubes described in Sections \ref{Sec:Tracking} and \ref{Sec:Exploration}. An overview of the different parts of the algorithm is shown in Figure \ref{fig:FlowChart}. The two main contributions of the paper can also be seen in the figure, as (i) improvements to tracking MPC using the blocks from Section \ref{Sec:Tracking} and (ii) the proposition of a novel application-oriented dual control scheme using the blocks from Section \ref{Sec:Exploration}. }

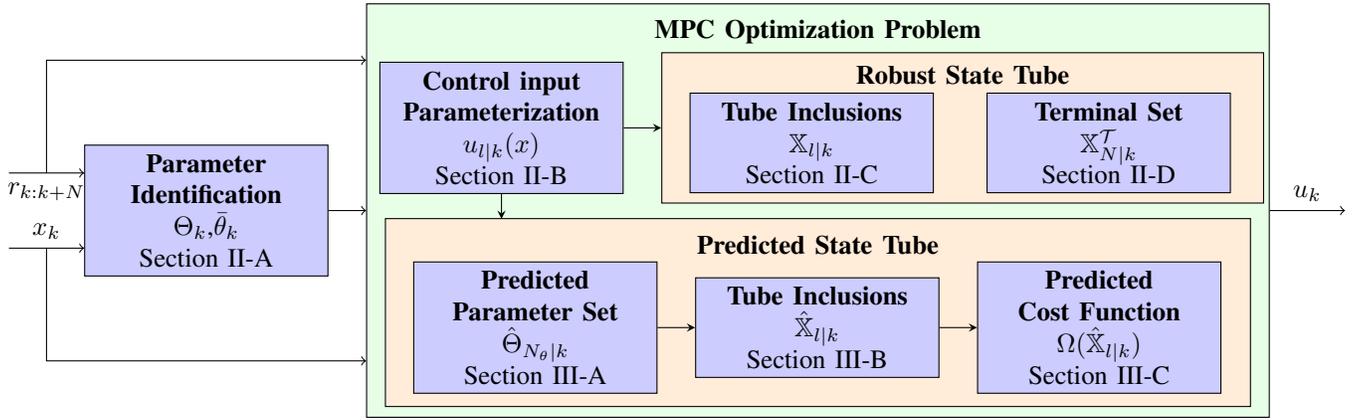
\begin{figure*}
	
	
	\begin{tikzpicture}
	
	\node [coordinate] (input) at (current page.center) {};		
	\node(placing) at ($ (input) + (0,0) $){};
	\node [rectangle,draw , fill=blue!20, right = 1cm of input, minimum width = 2cm, minimum height = 1cm ] (updPar)
	{ 
		\begin{minipage}{3cm}
		\centering
		\textbf{Parameter \\ Identification} \\
		$\Theta_k$,$\bar{\theta}_k$ \\
		Section \ref{Sec:ParameterIdentification}
		\end{minipage}
	};
	
	\node [rectangle,draw , fill=green!10,right = 0.5cm of updPar, minimum width = 12cm, minimum height = 5.5cm  ] (OptProb) {};
	\node[below=0.1cm] at (OptProb.north) {\textbf{MPC Optimization Problem}};
	
	\coordinate (x1) at ($(OptProb.west) + (1.8,1.1cm) $){}; 
	
	\node [rectangle,draw , fill=blue!20, minimum width = 3cm, minimum height = 1.5cm  ]  at (x1)  (SptEst) { 
		\begin{minipage}{3cm}
		\centering
		\textbf{Control input \\ Parameterization} \\
		$ u_{l|k} (x) $ \\
		Section \ref{Sec:ControlParameterization} 
		\end{minipage}
	};

	\coordinate (A1) at ($(updPar.west)+(-1,0.5)$){};
	\coordinate (A2) at ($(A1)+(0,-1)$){};
	
	\draw [->] (A1) -- node[anchor=north] {$r_{k:k+N}$} ($(updPar.west)+(0,0.5)$);
	\draw [->] (A2) -- node[anchor=south] {$x_k$} ($(updPar.west)+(0,-0.5)$);
	
	\coordinate (A3) at ($(OptProb.west)+(0,2)$){};
	\coordinate (A4) at ($(OptProb.west)+(0,-2)$){};
	\coordinate (A5) at ($(A1) + (0.5,0) $){};
	\coordinate (A6) at ($(A2) + (0.5,0) $){};
	\coordinate (A7) at ($(OptProb.east) + (1,0) $){};
	
	\draw [->] (A5) |-  (A3);
	\draw [->] (A6) |-  (A4);
	
	\draw [->] (updPar.east) --  (OptProb.west);
	
	\draw [->] (OptProb.east) -- node[anchor=south] {$u_k$}  (A7);

	\node [rectangle,draw , fill=orange!15,right = 0.5cm of SptEst.east, minimum width = 8cm, minimum height = 2cm  ] (RST) { 
	};
	\node [rectangle, below = 0.05cm of RST.north, minimum width = 5cm, minimum height = 0.5cm ] (RST_title) {
		\begin{minipage}{7cm}
		\centering
		\textbf{Robust State Tube} 
		\end{minipage}	
	};

	\draw [->,>=stealth] (SptEst.east) -- (RST.west);
	\draw [->,>=stealth] (SptEst.south) -- ($(SptEst.south) + (0,-0.33cm) $);

	\node [rectangle,draw , fill=blue!20, minimum width = 3cm, minimum height = 1cm  ] at  ($(RST.west) + (2,-0.2cm) $) (RST_TuInc) { 
		\begin{minipage}{3cm}
		\centering
		\textbf{Tube Inclusions} \\
		$ \mathbb{X}_{l|k} $ \\
		Section \ref{Sec:RobustTube}
		\end{minipage}
	};
	
	\node [rectangle,draw , fill=blue!20,right = 0.7cm of RST_TuInc.east, minimum width = 3cm, minimum height = 1cm  ]  (TermSet) { 
		\begin{minipage}{3cm}
		\centering
		\textbf{Terminal Set} \\
		$ \mathbb{X}_{N|k}^{\mathcal{T}} $ \\
		Section \ref{Sec:TerminalSets}
		\end{minipage}
	};
	
	\coordinate (x2) at ($(OptProb.south) + (0,1.4cm) $){}; 
	
	\node [rectangle,draw , fill=orange!15, minimum width = 11.5cm, minimum height = 2.5cm  ] at (x2) (PST) {    		 	 };
	\node [rectangle, below = 0.1cm of PST.north, minimum width = 5cm, minimum height = 0.5cm ] (PST_title) {
		\begin{minipage}{9.5cm}
		\centering
		\textbf{Predicted State Tube} \\
		\end{minipage}
	};
	\node [rectangle,draw , fill=blue!20, minimum width = 3cm, minimum height = 1cm  ] at  ($(PST.west) + (2,-0.2cm) $) (PredParSet) { 
		\begin{minipage}{3cm}
		\centering
		\textbf{Predicted \\ Parameter Set} \\
		$ \hat{\Theta}_{N_\theta|k}$ \\
		Section \ref{Sec:PredParSet}
		\end{minipage}
	};
	\node [rectangle,draw , fill=blue!20, right=0.5cm of PredParSet.east, minimum width = 3cm, minimum height = 1cm  ] (PST_TuInc) { 
		\begin{minipage}{3cm}
		\centering
		\textbf{Tube Inclusions} \\
		$\hat{\mathbb{X}}_{l|k}$\\
		Section \ref{Sec:PredictedTube}
		\end{minipage}
	};
	\node [rectangle,draw , fill=blue!20, right=0.5cm of PST_TuInc.east, minimum width = 3cm, minimum height = 1cm  ] (PredCost) { 
		\begin{minipage}{3cm}
		\centering
		\textbf{Predicted \\ Cost Function} \\
		$\Omega(\hat{\mathbb{X}}_{l|k})$\\
		Section \ref{Sec:PredictedCost}
		\end{minipage}
	};

	\draw [->,>=stealth] (PredParSet.east) -- (PST_TuInc.west);
	\draw [->,>=stealth] (PST_TuInc.east) -- (PredCost.west);
	
	\end{tikzpicture}
	
	\caption{ A schematic highlighting the different parts of the proposed dual adaptive MPC algorithm.  Each block shows the main variable, set or function computed in that block.  The measurements $x_k$ are obtained from the system, and the control inputs $u_k$ computed by the optimization problem are applied to the system.}
	\label{fig:FlowChart}
\end{figure*}

\subsection{Algorithm}
The decision variables \resp{of the MPC optimization problem} are 
\begin{equation}\label{eq:decisionVariables}
\gamma_k = \left\{ 
\begin{array}{l}
\bigl\{z_{l|k},\alpha_{l|k}\bigr\}_{l=0}^{N}, \: \bigl\{\hat{z}_{l|k},\hat{\alpha}_{l|k} \bigr\}_{l=0}^{N}, \\
\bigl\{v_{l|k},\: \{\Lambda^{j}_{l|k}\}_{j=1}^{q} , \: \: \{\hat{\Lambda}^{j}_{l|k}\}_{j=1}^{q}\bigr\}_{l=0}^{N-1} ,\\
\bigl\{\bar{x}_{l|k,k},\bar{u}_{l|k,k}\bigr\}_{l=0}^{N},\: \tilde{x}_{N|k},\tilde{u}_{N|k},\tilde{\alpha}_{N|k},\{\tilde{\Lambda}_{N|k}^{j}\}_{j=1}^{q}
\end{array} \right\},
\end{equation}
and the optimization problem can be written as 
\setlength\arraycolsep{2pt}
\begin{subequations}\label{eq:OptimizationProblem}
\begin{align}
\min_{\gamma_k} \quad  \Omega \bigr(\{\hat{\mathbb{X}}_{l|k}, \bar{x}_{l|k,k},\bar{u}_{l|k,k} \}_{l=0}^{N}, &\{v_{l|k}\}_{l=0}^{N-1},\tilde{u}_{N|k}  \bigr)  \nonumber \\
%
\text{s.t.} \quad (F+GK)z_{l|k} + Gv_{l|k} + \alpha_{l|k}\bar{f} &\le \mathbf{1}, \label{Opt1}\\
    -H_x z_{0|k} -\alpha_{0|k}\mathbf{1} &\le -H_x x_k ,		\label{Opt2}\\
    \Lambda_{l|k}^{j} h_{\theta_k} + H_x d_{l|k}^{j} -\alpha_{l+1|k} \mathbf{1} &\le -\bar{w},\label{Opt3}\\
    H_x D_{l|k}^{j} &= \Lambda_{l|k}^{j} H_{\theta}, \label{Opt4}\\
    -H_x \hat{z}_{0|k} - \hat{\alpha}_{0|k}\mathbf{1} &\le -H_x x_k, \label{Opt5}\\
    \hat{\Lambda}_{l|k}^{j} \hat{h}_{\theta_\resp{l|k}} + H_x \hat{d}_{l|k}^{j} -\hat{\alpha}_{l+1|k} \mathbf{1} &\le -\bar{w},\label{Opt6}\\
    H_x \hat{D}_{l|k}^{j} &= \hat{\Lambda}_{l|k}^{j} \hat{H}_{\theta_\resp{l|k}},\label{Opt7}\\
    \begin{bmatrix} 
    A(\bar{\theta}_k) & B(\bar{\theta}_k) \\C & 0
    \end{bmatrix} \begin{bmatrix}
    \bar{x}_{l|k,k}\\\bar{u}_{l|k,k} 
    \end{bmatrix} &= \begin{bmatrix} \bar{x}_{l+1|k,k} \\ r_{l|k}\end{bmatrix},\label{Opt8}\\
	(F+GK)\tilde{x}_{N|k} + G\Tilde{u}_{N|k} + \tilde{\alpha}_{N|k}\bar{f} &\le \mathbf{1}, \label{Opt9}\\
	\tilde{\Lambda}_{N|k}^{j} h_{\theta_k} + H_x d_{N|k}^{j} -\tilde{\alpha}_{N|k} \mathbf{1} &\le -\bar{w}, \label{Opt10}\\
	H_x D_{N|k}^{j} &= \tilde{\Lambda}_{N|k}^{j} H_{\theta},\label{Opt11}\\
    z_{N|k} = \bar{x}_{N|k,k},\quad \alpha_{N|k} &\le \tilde{\alpha}_{N|k}, \label{Opt12} \\
    \begin{bmatrix} 
    A(\bar{\theta}_k){-}I & B(\bar{\theta}_k) \\C & 0
    \end{bmatrix} \begin{bmatrix}
    \tilde{x}_{N|k,k}\\\bar{u}_{N|k,k} 
    \end{bmatrix} &= \begin{bmatrix} 0 \\ r_{k+N}\end{bmatrix}, \label{Opt13}\\
    \resp{ \alpha_{l|k} \ge 0,\: \hat{\alpha}_{l|k}  \ge 0,\: \tilde{\alpha}_{N|k} } &\resp{\: \ge 0,}\\
    \resp{ \Lambda^{j}_{l|k} \ge 0,\: \hat{\Lambda}^{j}_{l|k} \ge 0,\: \tilde{\Lambda}_{N|k}^{j}} &\resp{\: \ge 0,} \\
    &\forall j \in \mathbb{N}_{1}^{q}, \: l \in \mathbb{N}_{0}^{N-1}.  \nonumber
\end{align}
\end{subequations}
\resp{The cost function of the optimization problem is defined in \eqref{eq:MPCcost}. The constraints \eqref{Opt1} model the state and input constraints, \eqref{Opt2}-\eqref{Opt4} propagate the robust state tube according to Proposition \ref{Pr:SetDynamics}. The constraints \eqref{Opt5}-\eqref{Opt7} propagate the predicted state tube as defined in Proposition \ref{Pr:PredictedSetDynamics},  and \eqref{Opt8} computes the setpoint estimates used in the stage cost and input parameterization as stated in \eqref{eq:SetpointSubspace}. The terminal constraints from \eqref{eq:RPITerminalSet} are reproduced in \eqref{Opt9}-\eqref{Opt12}, and \eqref{Opt13} estimates the  terminal setpoint as defined in \eqref{eq:SetpointSubspaceTerm}. The optimization problem \eqref{eq:OptimizationProblem} is a nonconvex bilinear program due to the bilinear constraints \eqref{Opt6}, \eqref{Opt7}. 
} 

\begin{algorithm} [t]
	\caption{Reference tracking using dual AMPC}\label{Alg:AMPC} 
	\begin{algorithmic}[1]
	    \Statex \textbf{Input: $\bar{x}, r_k$}
		\Statex \textbf{Offline} Design $K,\mathbb{X}_0$ using Algorithm \ref{Alg:Offline}
		\Statex \textbf{Online}
		\State $ k\gets 1 $
		\Repeat 
		\State Obtain the measurement $ x_k $ 
		\State Construct $ \Delta_k $ according to \eqref{eq:SimpleNonfalsified}
		\State Update $ h_{\theta_k} $ using \eqref{eq:Theta_k_LP} and compute $ \bar{\theta}_k $ using \eqref{eq:ChebyshevMax}
		\State Solve optimization problem \eqref{eq:OptimizationProblem}
		\State Apply the control input  $ u_k = K (x_k-\bar{x}_{k,k}) + v_{0|k}$
		\State $ k \gets k+1 $
		\Until $k \gets T$
	\end{algorithmic}
\end{algorithm}
The dual AMPC algorithm that can be used to solve \eqref{eq:FHOCP} is described in Algorithm \ref{Alg:AMPC}. The prestabilizing gain $K$ and the state tube shape $ \mathbb{X}_0 $ are designed offline using Algorithm \ref{Alg:Offline}. The value of $h_{\theta_k} $ is initialized according to \eqref{eq:ParameterBounds}, and an initial guess is used for $\bar{\theta}_0$. 

\begin{remark}
\resp{An important property of tracking MPC algorithms is to guarantee offset free tracking. Nevertheless, such a guarantee} can only be proven asymptotically for systems affected by disturbances and model uncertainties \cite{maeder2009linear,pannocchia2004robust}. This cannot be achieved in a finite horizon setting, such as the one considered in this paper. However, the modifications proposed in Section \ref{Sec:Tracking} are not restricted to finite horizon problems, and can be used along with AMPC to achieve offset free tracking. For such an objective, the exploration based cost function must be replaced by a certainty equivalence one, based on the parameter $\bar{\theta}_k$ which is updated using a stable observer.  Additionally, the algorithm must ensure that the estimate satisfies $\bar{\theta}_k\rightarrow \theta^*$ asymptotically, using conditions such as persistence of excitation \cite{lorenzen2019}. 
\end{remark}

\subsection{Recursive feasibility}
One of the main objectives of the proposed dual AMPC method is to ensure robust constraint satisfaction while navigating the exploration-exploitation trade-off.  
\begin{theorem}\label{Th:RecFeas}
Let Assumptions \ref{As:FullRankEstimate}, \ref{As:Feedback} and \ref{As:Contractivity} be satisfied, and the optimization problem \eqref{eq:OptimizationProblem} be feasible for the initial time step $k=0$. The trajectories of the closed loop system using the dual AMPC algorithm \eqref{Alg:AMPC} will satisfy the constraints \eqref{eq:Constraints}, and the optimization problem \eqref{eq:OptimizationProblem} will remain feasible for all $k\in\mathbb{N}_1^{T}$. 
\end{theorem}
\begin{proof}
The recursive feasibility of \eqref{eq:OptimizationProblem} will be proven by induction. The feasibility at time step $k=0$ is assumed. Let the solution of \eqref{eq:OptimizationProblem} at time step $k$ be denoted as $\gamma_k^*$, where the superscript $^*$ is added to all the variables in \eqref{eq:decisionVariables}. Moreover, denote the robust state tube as $\{\mathbb{X}_{l|k}^*\}_{l=0}^{N}$ and the terminal set as $\mathbb{X}_{N|k}^{\pazocal{T}*}$. Using these variables, a feasible solution can be computed at the time step $k+1$. 

Because the robust state tube also satisfies the set-dynamics of the predicted state tube \eqref{eq:PredictedStateTube}, one only needs to find a feasible solution for the input variables $\{v_{l|k+1}\}_{l=0}^{N-1} $, robust state tube variables $  \{z_{l|k+1},\alpha_{l|k+1}\}_{l=0}^{N}, \{\{\Lambda^{j}_{l|k+1}\}_{j=1}^{q} \}_{l=0}^{N-1}, $ and terminal set variables $\tilde{x}_{N|k+1},\tilde{u}_{N|k+1},\tilde{\alpha}_{N|k+1},\{\tilde{\Lambda}_{N|k+1}^{j}\}_{j=1}^{q} $.  

Before the optimization problem \eqref{eq:OptimizationProblem} is solved at the next time step, the parameter estimate $\bar{\theta}_k$ and set $\Theta_k$ are updated to $\bar{\theta}_{k+1}$ and $\Theta_{k+1}$ respectively. Let $\bigl\{\bar{x}_{l|k+1,k+1},\bar{u}_{l|k+1,k+1}\bigr\}_{l=0}^{N}$ be a set of feasible solutions to the equations \eqref{eq:SetpointSubspace} and \eqref{eq:SetpointSubspaceTerm}  calculated at time step $k+1$, which exists due to Assumption \ref{As:FullRankEstimate}. Consider the sequence of input variables for $l \in \mathbb{N}_1^{N}$
\begin{equation}\label{eq:RecFeas1}
    v_{l-1|k+1} = v_{l|k}^* + K(\bar{x}_{l-1|k+1} - \bar{x}_{l|k}^*).
\end{equation}
This results in the control law
\begin{align}\label{eq:RecFeas2}
\begin{split}
    u_{l-1|k+1} (x) &= K(x - \bar{x}_{l-1|k+1}) + v_{l-1|k+1} \\
    &= K(x - \bar{x}_{l|k}^*) + v_{l|k}^*, \quad \forall l \in \mathbb{N}_1^{N}, \\
\end{split} 
\end{align}
which is equivalent to the optimal control law computed at time step $k$. Similarly, by defining $v_{N-1|k+1} = \tilde{u}_{N|k}^* + K(\bar{x}_{N-1|k+1} - \tilde{x}_{N|k}^*) $, the terminal control law computed at the time step $k$ can be used to compute a feasible solution for $v_{N-1|k+1} $.  Because the updated parameter set satisfies $\Theta_{k+1} \subseteq \Theta_k$, the robust state tube variables $  \{z_{l|k+1},\alpha_{l|k+1},\{\Lambda^{j}_{l|k+1}\}_{j=1}^{q} \}_{l=0}^{N-1} $ can be computed by setting
\begin{align}\label{eq:RecFeas3}
\begin{split}
    \mathbb{X}_{l-1|k+1} &= \mathbb{X}_{l|k}^*,  \quad \forall l \in \mathbb{N}_1^{N}.
\end{split}
\end{align}
The robust invariance of the set $\mathbb{X}_{N|k}^{\pazocal{T}*}$ implies that
\begin{equation}\label{eq:RecFeas4}
\mathbb{X}_{N|k}^{\pazocal{T}*} \supseteq  A_{\text{cl}}(\theta)\mathbb{X}_{N|k}^*  \oplus B(\theta)v_{N|k}^* \oplus \mathbb{W}.
\end{equation}
Thus, the values of the robust state tube variables \{$z_{N|k+1},\alpha_{N|k+1}$\} are computed by setting $\mathbb{X}_{N|k+1} = \mathbb{X}_{N|k}^{\pazocal{T}*}$. Additionally, robust invariance of $\mathbb{X}_{N|k}^{\pazocal{T}*}$ also implies that the terminal set variables $\tilde{x}_{N|k+1},\tilde{u}_{N|k+1},\tilde{\alpha}_{N|k+1},\{\tilde{\Lambda}_{N|k+1}^{j}\}_{j=1}^{q} $ can remain unchanged from the optimal values at time step $k$. 

Using the above procedure, a feasible solution to \eqref{eq:OptimizationProblem} can be constructed for the time step $k+1$, proving the recursive feasibility of Algorithm \ref{Alg:AMPC} by induction. Finally, the constraints \eqref{eq:Constraints} are always satisfied because the robust state tube lies within the constraint set, as shown in Proposition \ref{Pr:SetDynamics}. 
\end{proof}
\begin{remark}
It can be seen that the optimization problem \eqref{eq:OptimizationProblem} is recursively feasible due to the flexibility provided by the online terminal sets. This is because even if $r_{k+1}\neq r_k$, the optimizer can choose to center the terminal set at $\tilde{x}_{N|k+1} = \tilde{x}_{N|k}^* $ to ensure constraint satisfaction. This would not be possible with the offline design of terminal sets presented in \eqref{eq:OfflineTerminalSet} and \eqref{eq:OfflineTerminalConstraints}.
\end{remark}

\resp{
\subsection{Tube inclusion approximation}\label{Sec:TubeIncl_apx}
It can be seen that the number of constraints and variables involved in the optimization problem \eqref{eq:OptimizationProblem} depends on the number of vertices of $ \mathbb{X}_0 $, i.e., $q$. This means that the computational complexity of the optimization problem can grow combinatorially with the state dimension. This drawback of robust adaptive MPC has been addressed either by choosing a restrictive structure on the $ \Theta_k $ matrix \cite{kohler2019linear}, or using a distributed optimization approach by imposing structure on the system dynamics \cite{parsi2021distributed}. In this work, we propose a novel way to limit the optimization problem size by parameterizing the Lagrange multiplier variables $ \Lambda^{j}_{l|k}, \hat{\Lambda}^{j}_{l|k}, \tilde{\Lambda}^{j}_{N|k}$.

Define the matrices $ E_i = \begin{bmatrix}
A_i{+}B_iK & {-}B_iK & B_i
\end{bmatrix}$, $  \forall i \in \mathbb{N}_{0}^{p} $ and vectors $ e_{l|k} = \begin{bmatrix}
z_{l|k}^{\tr} & \bar{x}_{l|k,k}^{\tr} & v_{l|k}^{\tr}
\end{bmatrix}^{\tr} $, $  \forall l \in \mathbb{N}_{0}^{N} $. At each time step, before solving the MPC optimization problem, the matrices $ \Lambda^{j}_k \in \mathbb{R}^{n_x \times n_\theta}, \forall j\in \mathbb{N}_{1}^{q} $ are computed such that 
\begin{align}\label{eq:Lambdaj_apx}
\Lambda^j_k H_{\theta} &= H_x D(x^{j},Kx^{j}).
\end{align}
Then, the following quantities are defined
\begin{align}\label{eq:Lam_bar}
\underline{\Lambda}_k &= \min_{j\in \mathbb{N}_{1}^{q}} \Lambda^j_k,  \quad 
\bar{\lambda}_k = \max_{j\in \mathbb{N}_{1}^{q}} \Lambda^j_k h_{\theta_k} + H_x(A_0 + B_0 K)x^j ,
\end{align}
where the minimization and maximization are performed elementwise. Then, the following lemma reformulates the vertex-wise tube inclusion constraints so that they can be applied on the state tube variables.
\begin{proposition}\label{Prop:Lambda_apx}
Let the matrices $ \Lambda_{l|k} $ satisfy
\begin{subequations}\label{eq:Lambda_approx}
\begin{align}	
	&\Lambda_{l|k} + \alpha_{l|k} \underline{\Lambda}_k \ge 0, \label{eq:Lambda_approx1}\\
	\Lambda_{l|k} H_{\theta} &= H_x \begin{bmatrix}
	E_1 & E_2 & \ldots E_p
	\end{bmatrix} \otimes e_{l|k}, \label{eq:Lambda_approx2}\\
	\Lambda_{l|k} h_{\theta_k} &+ \alpha_{l|k} \bar{\lambda}_k + H_x E_0 e_{l|k} - H_x z_{l+1|k} - \alpha_{l+1|k}\mathbf{1} \le -\bar{w}. \label{eq:Lambda_approx3}
\end{align}
\end{subequations}
Then $ \Lambda^{j}_{l|k} = \Lambda_{l|k} + \alpha_{l|k} \Lambda_k^{j} $ satisfies the constraints \eqref{Opt3}, \eqref{Opt4} $ \forall j\in \mathbb{N}_{1}^{q} $. 
\end{proposition}
\begin{proof}
First, consider the constraints \eqref{Opt4}, where the term $ H_x D^j_{l|k} $ can be written as
\begin{align*}
H_x D^j_{l|k} =  H_x \begin{bmatrix}
E_1 & E_2 & \ldots E_p
\end{bmatrix} \otimes e_{l|k} + \alpha_{l|k}  H_x D(x^{j},Kx^{j})
\end{align*}
Thus, the vertex dependent part of the equality constraint can be separated from the vertex-independent part. The former is ensured by \eqref{eq:Lambdaj_apx}, and the latter by \eqref{eq:Lambda_approx2}. Then, consider the inequality constraints \eqref{Opt3}. Here, the term $ H_x d^j_{l|k} $ can be written as
\begin{align*}
H_x d^j_{l|k} =  H_x E_0 e_{l|k} -H_x z_{l+1|k} + \alpha_{l|k}  H_x D(x^{j},Kx^{j}).
\end{align*}
When the parameterization $ \Lambda^{j}_{l|k} = \Lambda_{l|k} + \alpha_{l|k} \Lambda^{j}_k $ is applied to \eqref{Opt3}, it can be seen that $ \alpha_{l|k} \bar{\lambda}_k $ bounds the effect of the vertex dependent terms.  Thus, the inequality \eqref{eq:Lambda_approx3} ensures \eqref{Opt3} for all the vertices of the state tube.
\end{proof}

Thus, by using the proposed parameterization, the robust state tube can be constructed using fewer optimization variables and constraints as shown in \eqref{eq:Lambda_approx}, instead of using the constraints \eqref{Opt3} and \eqref{Opt4} defined at each vertex of the state tube. A similar approach can be used to parameterize $ \tilde{\Lambda}^{j}_{N|k} $ as $ \tilde{\Lambda}_{N|k} + \alpha_{N|k} \Lambda^{j}_k $, thereby replacing the terminal constraints \eqref{Opt10}, \eqref{Opt11} with
\begin{subequations}\label{eq:LambdaN_approx}
\begin{align}	
	&\tilde{\Lambda}_{N|k} + \alpha_{N|k} \underline{\Lambda}_k \ge 0, \label{eq:LambdaN_approx1}\\
	\tilde{\Lambda}_{N|k} H_{\theta} &= H_x \begin{bmatrix}
	E_1 & E_2 & \ldots E_p
	\end{bmatrix} \otimes \tilde{e}_{N|k}, \label{eq:LambdaN_approx2}\\
	\tilde{\Lambda}_{N|k} h_{\theta_k} &{+} \alpha_{N|k} \bar{\lambda}_k {+} H_x E_0 \tilde{e}_{N|k} - H_x \tilde{x}_{N|k} - \alpha_{l+1|k}\mathbf{1} \le -\bar{w}, \label{eq:LambdaN_approx3}
\end{align}
\end{subequations}
where $ \tilde{e}_{N|k} = \begin{bmatrix}
\tilde{x}_{N|k}^{\tr} & \tilde{x}_{N|k}^{\tr} & \tilde{u}_{N|k}^{\tr}
\end{bmatrix}^{\tr} $.

Finally, the Lagrange multipliers defining the predicted state tube are parameterized as  
\begin{align}\label{eq:Lambda_hat_par}
\hat{\Lambda}^{j}_{l|k} = \begin{bmatrix}
\hat{\Lambda}_{l|k}+\hat{\alpha_{l|k}}\Lambda^j_k & \hat{\Lambda}^{\Delta}_{l|k}
\end{bmatrix}.
\end{align}
It can be seen that the first $ n_{\theta} $ columns of $ \hat{\Lambda}^{j}_{l|k} $ are parameterized using the precomputed $ \Lambda^j_k $, while the predicted constraints in $ \hat{\Delta}_{l|k} $ will be multiplied by $ \hat{\Lambda}^{\Delta}_{l|k} $. Using \eqref{eq:Lambda_hat_par}, the tube inclusion constraints \eqref{Opt6} and \eqref{Opt7} can be replaced by the sufficient conditions
\begin{subequations}\label{eq:Lambda_hat_approx}
	\begin{align}	
	&\hat{\Lambda}_{l|k}+\hat{\alpha_{l|k}} \underline{\Lambda}_k \ge 0, \quad \hat{\Lambda}^{\Delta}_{l|k} \ge 0,	 \label{eq:Lambda_hat_approx1}\\
	\begin{bmatrix}
	\hat{\Lambda}_{l|k} & \hat{\Lambda}^{\Delta}_{l|k}
	\end{bmatrix}
	 \hat{H}_{\theta_{l|k}} &= H_x \begin{bmatrix}
	E_1 & E_2 & \ldots E_p
	\end{bmatrix} \otimes \hat{e}_{N|k}, \label{eq:Lambda_hat_approx2}\\
	\begin{bmatrix}
	\hat{\Lambda}_{l|k} & \hat{\Lambda}^{\Delta}_{l|k}
	\end{bmatrix}
	\hat{h}_{\theta_{l|k}} &{+} \hat{\alpha}_{l|k} \bar{\lambda}_k {+} H_x E_0 \hat{e}_{l|k} \nonumber \\
	& \quad - H_x \hat{z}_{l+1|k} - \hat{\alpha}_{l+1|k}\mathbf{1} \le -\bar{w}. \label{eq:Lambda_hat_approx3}
	\end{align}
\end{subequations}
The sufficiency of \ref{eq:Lambda_hat_approx} follows from a similar proof as in Proposition \ref{Prop:Lambda_apx}. Thus the number of nonconvex constraints in the optimization problem do not increase with the number of vertices in the state tube set $ \mathbb{X}_0 $.  

The amount of conservatism induced by the proposed sufficient conditions depends on the matrices $ \Lambda^{j}_{k} $ computed according to \eqref{eq:Lambdaj_apx}. Since the matrix $ H_{\theta} $ describes the bounded set $ \Theta_k $, it has more rows than columns. Thus, the choice of $ \Lambda^{j}_{k} $ is non-unique based on \eqref{eq:Lambdaj_apx} alone. The quantities $ \underline{\Lambda}_k , \bar{\lambda}_{k}$ can be used to optimize over the feasible values of $ \Lambda^{j}_{k} $. As seen in \eqref{eq:Lam_bar} and \eqref{eq:Lambda_approx1}, negative values in $ \Lambda^{j}_{k} $ result in a minimum positive magnitude for $ \Lambda_{l|k} $, thereby growing the state tube quickly as seen in \eqref{eq:LambdaN_approx3}. In addition, large values for $ \bar{\lambda}_{k} $ directly influence the growth of the state tube in \eqref{eq:LambdaN_approx3}. The following optimization problem is solved at each time step to compute the matrices $ \Lambda^{j}_{k} $:

\begin{align}\label{eq:Lambda_Opt}
\begin{split}
\displaystyle\min_{\{\Lambda^{j}_{k}\}} \quad ||\bar{\lambda}_k||_2 &+ \mu_{\theta}||\Lambda^j_k||_{\max} , \\
\text{s.t.} \qquad \Lambda^j_k H_{\theta} &= H_x D(x^{j},Kx^{j}), \\
\bar{\lambda}_k \ge \Lambda^j_k &h_{\theta_k} + H_x(A_0 + B_0 K)x^j, \quad \forall j\in \mathbb{N}_{1}^{q},
\end{split}	
\end{align}
where $  \mu_{\theta} $ is a tuning parameter which must be chosen such that the precomputed $ \Lambda^j_k$ matrices do not have large negative numbers or result in large $ \bar{\lambda}_k $.
}

\section{Simulation Studies}\label{Sec:Simulations}
In this section, we illustrate the advantages of the proposed dual adaptive MPC scheme using comparative simulations \resp{on two systems. The first example considers a 2 state, 2 input system with 2 uncertain parameters. The second example considers a 6 state, 3 input system with 5 uncertain parameters in the model. The code to simulate both the examples is available in the online repository \cite{DAMPC_repo}. }

\subsection{Two state example}	\label{Sec:2StateEx}
For the first example, the model of the system is parameterized by the matrices
\begin{equation}\label{eq:ExampleSystem}
\begin{array}{l l l}
A_0 = \begin{bmatrix} 0.85 &  0.5 \\ 0.2 & 0.7 \end{bmatrix}, & A_1 = \begin{bmatrix} 0.1 &  0 \\ 0 & 0.2 \end{bmatrix}, &A_2 = \begin{bmatrix} 0 &  0 \\ 0 & 0 \end{bmatrix}, \\
B_0 = \begin{bmatrix} 1 & 0.4\\ 0.2& 0.6 \end{bmatrix},  & B_1 = \begin{bmatrix} 0 &  0 \\ 0 & 0 \end{bmatrix},  & B_2 = \begin{bmatrix} 0 & 0.2\\ 0 & 0.35 \end{bmatrix},  \\
\end{array} 
\end{equation}
where the state and input matrices are affected by different uncertain parameters. The outputs to be tracked are the two state variables, and thus a target state trajectory is given by $r_k$. In addition, because $m=n_y=2$ the constraint \eqref{eq:SetpointSubspace} defines a unique input setpoint. The uncertainty is defined by the parameter set $\Theta = \{\theta\in\mathbb{R}^{2}|\: ||\theta||_\infty \le 1 \}$. Note that the parameter $[\theta]_2$ is affecting the second input channel only, and its relative effect on $B(\theta)$ is larger compared to the effect of $[\theta]_1$ on $A(\theta)$. The constraint set $\mathbb{Z}$ is defined as 
\begin{align*}
	\mathbb{Z} &= \left\{(x,u){\in} \mathbb{R}^{2\times 2}\left|\: \begin{array}{rl}
	||x||_\infty &\le 3 \\
	||u||_\infty &\le 2 \\ 
	\end{array}\right.  \right\},
\end{align*}
and the disturbance set is $\mathbb{W}:= \{w\in\mathbb{R}^2|\: ||w||_\infty \le 0.1\} $. The cost function is defined by the matrices $Q = 2I_2$ and $ R = I_2$. 

The controller using Algorithm \ref{Alg:AMPC} is referred to as dual AMPC (DAMPC) in this section. The performance of DAMPC is compared to the adaptive MPC algorithm proposed in \cite{lorenzen2017adaptive}, referred to as passive AMPC (PAMPC). This is because PAMPC updates the parameter set using set membership, but does not model the dual effect of the control input in the optimization problem. This makes the uncertainty learning passive. The PAMPC controller was proposed in \cite{lorenzen2017adaptive} for a regulation task. For the sake of comparison, the algorithm is modified to enable setpoint tracking using the offline terminal set approach presented in Section \ref{Sec:TerminalSets}.  In addition, the PAMPC controller also uses \eqref{eq:SetpointSubspace} to estimate the setpoints in the MPC cost function. Both the controllers are used to track a piecewise constant reference trajectory with five different setpoints for $T=100$ time steps, and use a prediction horizon of $N=8$ steps. \resp{The DAMPC controller used a lookahead horizon $ N_{\theta}=8 $ steps. The control gain $ K $ is designed such that the unit hyperbox in 2 dimensions satisfies Assumption \ref{As:Contractivity}, with $ \lambda_c = 0.96 $. This implies $ n_x{=}4$ and $ q{=}4 $ for the chosen $ \mathbb{X}_0 $.} The simulations are performed for 50 different realizations of the parameter $\theta^*\in\Theta$. For each realization of $\theta^*$, 4 simulations with different disturbance sequences $\{w_k\}_{k=0}^{T} \in \mathbb{W}$ are performed. The values of $\theta^*$ and $\{w_k\}_{k=0}^{T}$ are generated at random with a uniform distribution, \resp{ and the same parameters and disturbance sequences are used in closed loop for the two controllers}.  \resp{Moreover, a third controller is also designed with perfect knowledge of $ \theta^* $. This controller, called the FHOCP controller, computes an input sequence for the whole remaining control horizon (from $k$ to $ T $) instead of the a shorter prediction horizon. Moreover, the controller also constructs a state tube using $ \mathbb{X}_0 $ to be robust against the additive disturbance. This controller is used to quantify the loss of performance caused by the exploratory inputs applied to the system in order to identify the parameters. } All the controllers are initialized with the same initial guess $\bar{\theta}_0 = [0.1,0.1]\tr$, and the initial state of the system is at the origin. 

The closed loop costs achieved by the controllers for the 200 simulations are compared in Fig. \ref{fig:cost}, where the frequency of occurrence of the costs is shown on y-axis. The average closed loop cost achieved by the \resp{FHOCP controller is 75.5,} DAMPC  is \resp{81.5} and PAMPC is \resp{172.1. It can be seen from the distribution of costs that the FHOCP controller results in a slightly improved performance over DAMPC, especially in the simulations with low closed loop costs.} In order to provide more insight on the variance of the costs, the distribution of the costs achieved by PAMPC and DAMPC is plotted as a function of \resp{$ \theta^* $} in Fig. \ref{fig:ScatterCost}, where the costs are averaged over the 4 realizations of $w_k$ for each $\theta^*$.
\setlength{\dbltextfloatsep}{6pt}
\begin{figure}[t]
    \centering
    \includegraphics{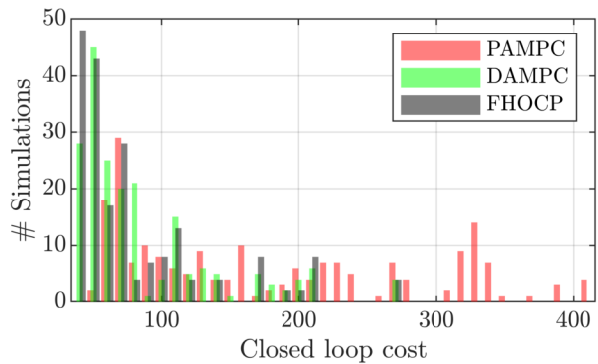}
    \caption{Comparison of closed loop costs over 200 simulations with different realizations of uncertainty and disturbance.} 
    \label{fig:cost}
\end{figure}
\setlength{\dbltextfloatsep}{6pt}
\begin{figure}[t]
    \centering
    \includegraphics{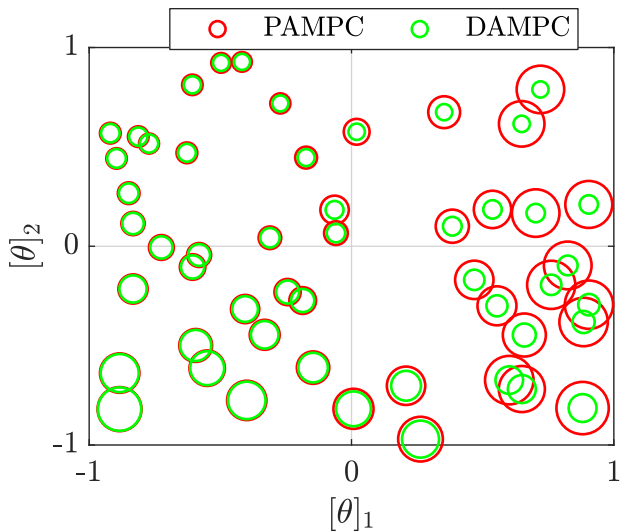}
    \caption{Distribution of costs as a function of the true parameter, averaged over the disturbance realizations. The center of each circle represents the true parameter, and the area of the circle is proportional to the averaged cost.} 
    \label{fig:ScatterCost}
\end{figure}
In this figure, each circle is centered at the true parameter and its area represents the averaged closed loop cost for that parameter.  The figure can be analyzed in four quadrants. In the upper-left quadrant, both controllers show similar performance and relatively small closed loop costs. This represents a region where passive exploration is sufficient to reduce the uncertainty and track the given trajectory. In the upper- and lower-right quadrants, the DAMPC algorithm shows noticeably smaller costs. This indicates that true systems from this region of parameter space benefit from the active uncertainty learning. Finally, most systems from the lower-left quadrant have relatively higher costs for both algorithms. This is because some of the true input setpoints $u_k^*$ lie outside the feasible region for these systems. 

\resp{
The closed loop trajectories of the system for one realization of the true parameters ($\theta^* = [ {-}0.06 ,  0.07]\tr$) are shown in Fig. \ref{fig:traj}. Only the first 60 time steps of the simulation are shown, because both DAMPC and PAMPC controllers identify the parameters and have a similar response after this time for this system. For the DAMPC controller, three different values of $ N_\theta $ were used, and the trajectories are labeled as $D_2, D_5, D_8$ for $ N_{\theta} {=} 2,5,8 $ respectively. It can be seen that all the DAMPC controllers excite the system before $ t=20 $, although the reference to be tracked is 0 and the system is already close to the origin. This is an exploratory input signal, which the controller activates immediately after a change of reference setpoint is seen in the prediction horizon. The $ D_2 $ controller uses large input signals to identify the plant, since it has assumes a short experiment length for identification. The $ D_5 $ and $ D_8 $ controllers use smaller excitation signals, and achieve lower closed loop costs. The PAMPC controller does not use any exploratory excitation in the initial part of the trajectory, as expected. This means that the uncertainty set used by the PAMPC controller when the setpoint change occurs is difficult to ensure robustness compared to the one used by the DAMPC controllers, and it results in a poor setpoint tracking performance. Moreover, it can be seen that due to the large uncertainty in the second input channel, the PAMPC controller does not fully use $[u]_2$ until $ t=34 $ when passive exploration has considerably reduced the uncertainty. }


\resp{The simulations were performed using a Intel Xeon Gold 5118 processor with 2GB RAM. All the optimization problems were implemented using YALMIP \cite{Lofberg2004}. The optimization problem of the PAMPC algorithm is a linear program, and was solved using MOSEK \cite{mosek}. The average computation times for the solving the PAMPC optimization problem was $ 0.5s $. The DAMPC optimization problem is nonconvex, due to the bilinear equalities and inequalities in \eqref{eq:OptimizationProblem}. This optimization problem was solved using IPOPT \cite{wachter2006implementation}, which uses an interior-point based line-search algorithm. The optimization problems were warm-started with the solution obtained at the previous time step. The average computation times for solving the DAMPC algorithm for $ N_{\theta} = 2,5,8 $ were $ 0.5s, 1.1s, 1.8s $ respectively. The average values of the relative increase in closed loop cost, of DAMPC compared to FHOCP were 16.3\%, 11.6\%, 10.8\% for $ N_\theta = 2,5,8$ respectively. Thus, it can be seen that increasing $ N_\theta $ improves performance and also increases the computational cost. }

\setlength{\dbltextfloatsep}{6pt}
\begin{figure}[t]
	\centering
	\includegraphics{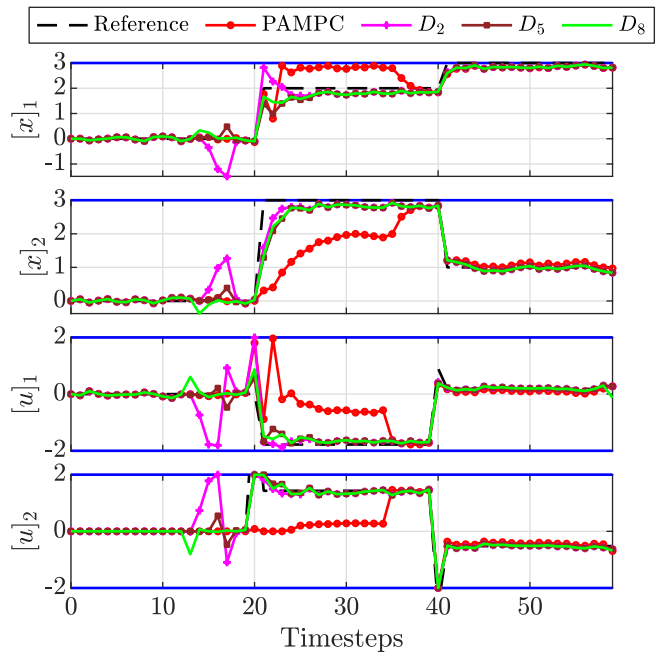}
	\caption{Comparison of trajectories for PAMPC and DAMPC (with $ N_{\theta} =2,5,8 $) for the first 60 timesteps. The true state and input setpoints are shown as a dashed line, and the constraints are shown in blue. Recall that the controllers do not have knowledge of the true input setpoints.} 
	\label{fig:traj}
\end{figure}

\resp{	
\subsection{Six state example}		\label{Sec:6StateEx}
As a second simulation example, we consider a mass-spring damper system with three masses $ m_1, m_2, m_3 $, each weighing 1kg and connected along a line by two springs and dampers. Using the position and velocity of each mass as the state of the system, the state matrix can be written as
\begin{align*}
A = \begin{bmatrix}
0 & 1 & 0  & 0 & 0 & 0 \\
{-}k_{12} & {-}c_{12} & k_{12} & c_{12} & 0 & 0 \\ 
0 & 0 & 0  & 1 & 0 & 0 \\
k_{12} & c_{12} & {-}k_{12}{-}k_{23} & {-}c_{12}{-}c_{23} & 0 & 0 \\
0 & 0 & 0  & 0 & 0 & 1 \\
0 & 0 & k_{23} & c_{23} & -k_{23} & -c_{23}
\end{bmatrix}, 
\end{align*}
and the input matrix as
\begin{align*}
B = 
\begin{bmatrix}
0 & F_g & 0 & 0 & 0 & 0 \\
0 & 0 & 0 & F_g & 0 & 0 \\
0 & 0 & 0 & 0 & 0 & F_g \\
\end{bmatrix}^\tr
\end{align*}
where $ k_{ij}$ and $ c_{ij} $ represent the spring constant and damping coefficient of the elements connecting masses $ i $ and $ j $ respectively. In addition, $ F_g $ is an actuator gain on the control input to give the force acting on each mass. The spring constants have the nominal values $ k_{12} = 3.2 Nm^{-1}, k_{23} =5.8 Nm^{-1} $ and an uncertainty of $ \pm 10\% $ of their nominal value. The damping coefficients are given as $ c_{12} = 2.3 Nsm^{-1}, c_{23} = 4.5 Nsm^{-1}  $ with a $ \pm 5\% $  uncertainty. The actuator gain $ F_g = 6.4 $ with a $ \pm 7\% $ uncertainty. The system dynamics are discretized using Euler discretization with a sampling time of $ 0.1s $ to preserve the uncertainty structure. The resulting model has 6 states (positions and velocities of the masses) and 3 inputs. Additionally, an additive disturbance with a maximum magnitude of 0.05 is acting on each state of the system. The outputs to be tracked are the positions of the masses, and a reference trajectory of 75 timesteps is considered. The states and inputs of the system are constrained to be within the bounds $ \pm 5 $. The cost function is defined by the matrices $ Q = 2 I_3  $ and $ R = I_3 $.

The tracking performance with PAMPC and DAMPC controllers is shown in Figure \ref{fig:traj_6s}, where the highest and lowest values of the positions at each timestep over 150 simulations (50 random realizations of $ \theta^* $, each with 3 random realizations of $w_k$) are plotted. The DAMPC controller has a prediction horizon of $ N=6 $ steps and lookahead horizon $ N_{\theta} = 4 $. The set $ \mathbb{X}_0 $ is designed using Algorithm \ref{Alg:Offline}, and is defined by 24 hyperplanes and 322 vertices. The system is affected by 5 parameters (2 springs, 2 dampers and 1 actuator gain), and the uncertainty set $ \Theta_k $ is represented as a hyperbox in 5 dimensions. Due to the large size of the resulting optimization problem, it was chosen to use the tube inclusion approximations proposed in Section \ref{Sec:TubeIncl_apx}. The PAMPC controller also uses a prediction horizon of $ N=6 $ steps.  It can be seen that the DAMPC controller tracks the reference better, as shown by the closeness of the upper and lower bounds to the reference trajectory. 

The closed loop cost achieved by the DAMPC controller is 85.4, and that of PAMPC is 135.6. This improved performance comes at the cost of computational complexity. The simulations were performed on a Intel Xeon E3-1585Lv5 processor with 6GB RAM. The average computation time to solve the DAMPC optimization problem was $ 110.7s $ and that of the PAMPC optimization problem is $ 105.4 s$. Note that the DAMPC optimization problem was solved using the tube inclusion approximations proposed in Section \ref{Sec:TubeIncl_apx}, while the PAMPC optimization problem was setup using tube inclusion constraints at each vertex. 



\setlength{\dbltextfloatsep}{6pt}
\begin{figure}[t]
    \centering
    \includegraphics{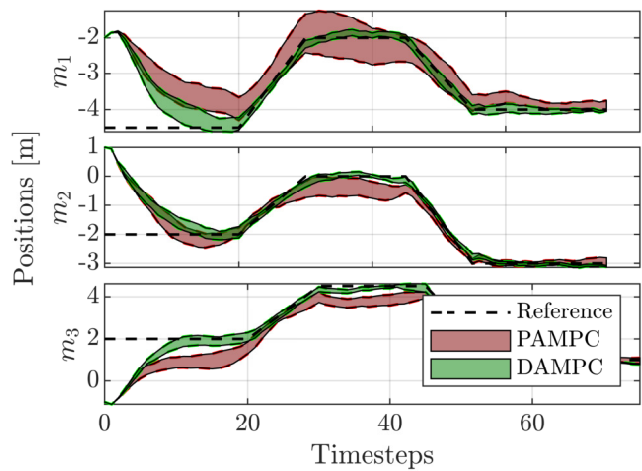}
    \caption{Reference tracking performance of the DAMPC and PAMPC controllers when applied to the mass spring damper example. For each mass, the highest and lowest values of positions at each timestep over 150 simulations are plotted. } 
    \label{fig:traj_6s}
\end{figure}
}

\section{Conclusions}
In this paper, we \resp{propose} an explicit dual control approach to approximate a finite horizon optimal tracking problem for uncertain linear systems. In order to learn the uncertainty, an adaptive MPC scheme is used. A tracking formulation is proposed to enable the computation of online terminal sets which are robustly invariant. This is particularly beneficial for adaptive MPC algorithms, since the uncertainty set is reduced online. In addition, the cost function uses a predicted state tube in order to model the dual effect of the control input, and the contractivity of the proposed terminal sets is used to approximate the cost-to-go of the optimal control problem. The proposed algorithm guarantees recursive feasibility in the face of changing reference trajectories and parameter estimates. \resp{The resulting optimization problem is nonconvex due to terms in the constraints, and its size can grow combinatorially with the state dimension of the system. A novel tube inclusion approximation is proposed to prevent this growth, which can also be applied to other adaptive MPC schemes in the literature. }
\resp{
Two simulation studies highlight
}the performance improvement achieved by the proposed algorithm over an adaptive MPC approach which does not model the dual control effect.

An interesting direction for future research is to extend the algorithm for time varying systems, where exploration would be beneficial throughout the control horizon. This is not the case for time invariant systems, since exploration can be stopped once the uncertainty set has been reduced to a small size. \resp{Another promising direction is to implement the algorithm using ellipsoidal sets to parameterize the state tube, which would simplify the offline design of contractive sets. }

\bibliographystyle{IEEEtran}
\bibliography{biblio_RAMPC_TAC}

\begin{IEEEbiography}[{\includegraphics[width=1in,height=1.25in,clip,keepaspectratio]{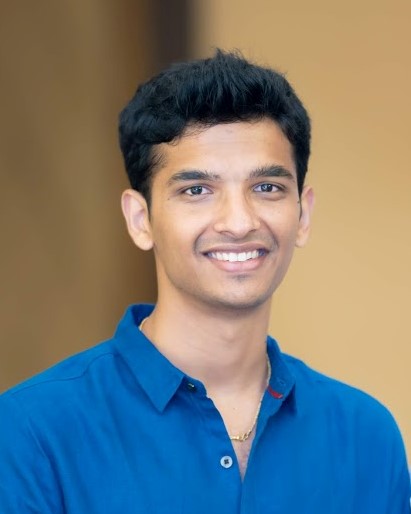}}]{Anilkumar Parsi} completed his B.Tech (Hons) in Mechanical Engineering at the Indian Institute of Technology Madras in 2016. He received an M.Sc. in Mechanical Engineering from ETH Z\"{u}rich in 2019.  He is currently a Ph.D. student in the Automatic Control Laboratory at ETH Z\"{u}rich. His research interests include model predictive control, robust and learning based control.
\end{IEEEbiography}

\begin{IEEEbiography}[{\includegraphics[width=1in,height=1.25in,clip,keepaspectratio]{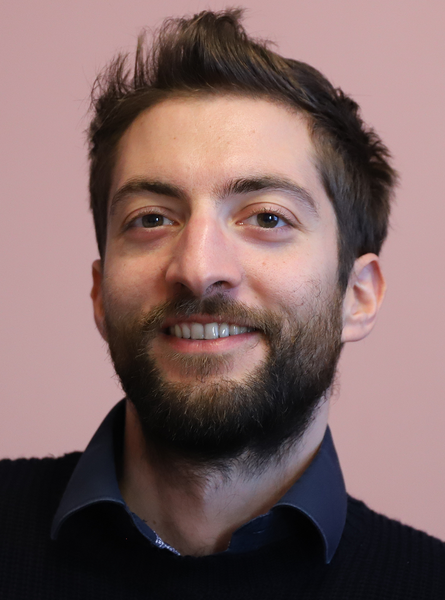}}]
	{Andrea Iannelli} completed the B.Eng. and M.Sc. degrees in Aerospace Engineering at the University of Pisa and he received his PhD from the University of Bristol. In his doctoral studies he worked on the reconciliation between robust control and dynamical systems theory, with application to linear and nonlinear aeroelastic problems. He is currently a postdoctoral researcher in the Automatic Control Laboratory at ETH Z\"{u}rich. His research interests include modeling, analysis and synthesis of uncertain control systems, data-driven methods, and dynamical systems theory.
\end{IEEEbiography}

\begin{IEEEbiography}[{\includegraphics[width=1in,height=1.25in,clip,keepaspectratio]{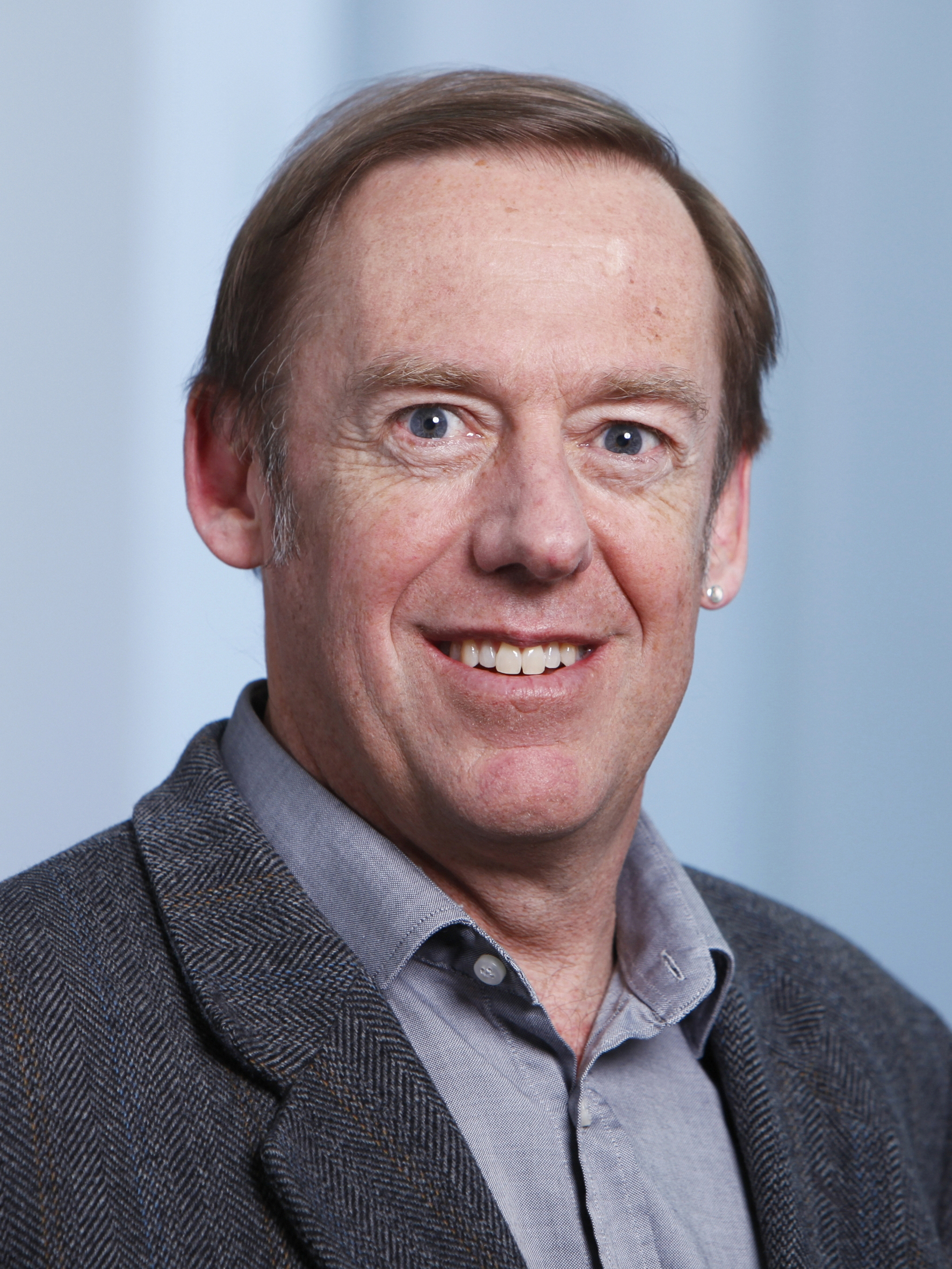}}]{Roy S. Smith} is a professor of Electrical
	Engineering at the Swiss Federal Institute of Technology (ETH),
	Zürich. Prior to joining ETH in 2011, he was on the faculty of the
	University of California, Santa Barbara, from 1990 to 2010. His
	Ph.D. is from the California Institute of Technology (1990) and his
	undergraduate degree is from the University of Canterbury (1980) in
	his native New Zealand. He has been a long-time consultant to the NASA
	Jet Propulsion Laboratory and has industrial experience in automotive
	control and power system design.  His research interests involve the
	modeling, identification, and control of uncertain systems. Particular
	control application domains of interest include chemical processes,
	flexible structure vibration, spacecraft and vehicle formations,
	aerodynamic control of kites, automotive engines, Mars
	aeromaneuvering entry design, building energy control, and
	thermoacoustic machines. He is a Fellow of the IEEE and the IFAC, an
	Associate Fellow of the AIAA, and a member of SIAM. 

\end{IEEEbiography}

\end{document}